\documentclass{article}

\usepackage{epsfig}
\usepackage{amsfonts}
\usepackage{amssymb}
\usepackage{amsmath}   
\usepackage{amsthm}
\usepackage{latexsym}
\usepackage{enumerate}
\usepackage[subnum]{cases}
\usepackage{hyperref}
\usepackage{appendix}

\newtheorem{thm}{Theorem}[section]
\newtheorem{lemma}[thm]{Lemma}
\newtheorem{cor}[thm]{Corollary}

\newtheorem{defin}[thm]{Definition}

\newtheorem{rmk}[thm]{Remark}

\numberwithin{equation}{section}

\DeclareMathOperator*{\hess}{\nabla^2}
\DeclareMathOperator*{\tr}{tr}

\DeclareMathOperator*{\dist}{dist}

\DeclareMathOperator*{\hessxy}{\nabla^2_{x,y}}

\DeclareMathOperator*{\hessx}{\nabla^2_{x}}
\DeclareMathOperator*{\hessy}{\nabla^2_{y}}

\renewcommand{\epsilon}{\varepsilon}

\newcommand{\R}{\mathbb{R}}

\newcommand{\omstar}{\Omega^*}
\newcommand{\omclose}{\overline{\Omega}}
\newcommand{\omstarclose}{\overline{\omstar}}
\newcommand{\ombdry}{\partial \Omega}
\newcommand{\omstarbdry}{\partial \omstar}
\newcommand{\epstar}{\epsilon^*}
\newcommand{\omstarbdryepstar}{\Gamma^*_{\epstar}}

\newcommand{\Omtil}{\Omega_1}
\newcommand{\Omtilbdry}{\partial\Omtil}
\newcommand{\omep}{\Omega_\epsilon(x_0)}
\newcommand{\omepbdry}{\Gamma_\epsilon}

\newcommand{\nustar}{\nu^*}

\newcommand{\gradphi}{\nabla\phi}
\newcommand{\gradu}{\nabla u}
\newcommand{\gradustar}{\nabla u^*}

\newcommand{\omtil}{Y^{-1}(\omstarbdryepstar)}
\newcommand{\omtilbdry}{\partial\left(\omtil\right)}
\newcommand{\Mtil}{\tilde{M}}

\newcommand{\Btil}{\tilde{B}}
\newcommand{\udot}{\dot{u}}
\newcommand{\ustardot}{\dot{u^*}}
\newcommand{\vdot}{\dot{v}}
\newcommand{\wdot}{\dot{w}}
\newcommand{\phidot}{\dot{\phi}}
\newcommand{\hstar}{h^*}
\newcommand{\hstarbar}{\bar{h}^*}
\renewcommand{\hbar}{\bar{h}}
\newcommand{\Gbar}{\bar{G}}
\newcommand{\Bdel}{B_\delta}
\newcommand{\delstar}{\delta^*}
\newcommand{\delbar}{\delta_1}
\newcommand{\delstarbar}{\delstar_1}
\newcommand{\deltil}{\delstar_2}
\newcommand{\inner}[2]{\ensuremath{\langle#1, #2\rangle}}
\newcommand{\Dp}[1][k]{\ensuremath{D_{p_#1}}}
\newcommand{\Dpp}[2][k]{\ensuremath{D_{p_#1p_#2}}}
\newcommand{\Dxx}[2][k]{\ensuremath{D_{x_#1x_#2}}}
\newcommand{\Dxp}[2][k]{\ensuremath{D_{x_#2p_#1}}}
\newcommand{\Dx}[1][k]{\ensuremath{D_{x_#1}}}
\newcommand{\Fx}[1][k]{\ensuremath{F_{x_#1}}}
\newcommand{\Fxx}[2][k]{\ensuremath{F_{x_#1x_#2}}}
\newcommand{\Fp}[1][k]{\ensuremath{F_{p_#1}}}
\newcommand{\Fpp}[2][k]{\ensuremath{F_{p_#1p_#2}}}
\newcommand{\Gp}[1][k]{\ensuremath{G_{p_#1}}}
\newcommand{\Gpp}[2][k]{\ensuremath{G_{p_#1p_#2}}}
\newcommand{\Fxp}[2][k]{\ensuremath{F_{x_#2p_#1}}}
\newcommand{\pdiff}[2][x]{\ensuremath{\frac{\partial}{\partial#1_{#2}}}}

\newcommand{\ftil}{\hat{f}}
\newcommand{\util}{\hat{u}}
\newcommand{\utildot}{\dot{\util}}
\newcommand{\wtil}{\hat{w}}
\newcommand{\vtil}{v_1}
\newcommand{\ytil}{\tilde{y}}
\newcommand{\lambdatil}{\tilde{\lambda}}

\newcommand{\vvec}{\vec{v}}
\newcommand{\ttil}{t_1}
\newcommand{\tmax}{t_{max}}
\newcommand{\Tinv}{(T)^{-1}}

\newcommand{\phibar}{\bar{\phi}}
\newcommand{\tbar}{t_2}
\newcommand{\taustar}{\tau^*}
\newcommand{\etadot}{\dot{\eta}}

\author{Jun Kitagawa}
\title{A parabolic flow toward solutions of the optimal transportation problem on domains with boundary}
\begin{document}
\maketitle

\begin{abstract}
We consider a parabolic version of the mass transport problem, and show that it converges to a solution of the original mass transport problem under suitable conditions on the cost function, and initial and target domains.
\end{abstract}

\section{Introduction}
We are concerned in this paper with solutions to the optimal transport problem, which reads as follows. 

Given two domains $\Omega$ and $\omstar$, and two probability measures $\mu$ and $\nu$ defined on them, along with a real valued cost function $c$ defined on $\omclose \times \omstarclose$, we wish to find a measurable mapping $T: \Omega \to \omstar$ satisfying $T_\# \mu = \nu$ (defined by $T_\# \mu (E)= \mu(\Tinv(E))$ for all measurable $E\subset \omstar$) such that 
\begin{equation*}
\int_\Omega c(x, T(x))d\mu = \max_{S_\#\mu=\nu}\int_\Omega c(x, S(x))d\mu.
\end{equation*}
Under mild conditions on $c$ and the measures $\mu$ and $\nu$, it is known that the solution to this problem exists. Additionally, if $\Omega$ and $\omstar$ are subsets of $R^n$, and $\mu$ and $\nu$ are absolutely continuous with respect to Lebesgue measure, $T$ can be determined from a scalar valued potential function satisfying the following equation in an appropriately weak sense:
\begin{equation*}
\begin{cases}{}
\det{(\hess{u(x)} - A(x, \gradu(x)))}=B(x, \gradu(x)),&x\in\Omega\\
T(\Omega)=\omstar&
\end{cases}
\end{equation*}
where $A$ is a matrix valued function and $B$ is a scalar valued function defined from $c$ and the densities of the two measures.

Under certain conditions on the domains, cost, and measures, the interior regularity of the potential $u$ has been shown by Ma, Trudinger, and Wang in \cite{MR2188047}, and global regularity by Trudinger and Wang in the subsequent \cite{MR2512204}. In \cite{MR2008688}, Schn{\"u}rer and Smoczyk analyze a parabolic flow which is close to the corresponding optimal transport problem with the cost $c(x, y)=\lvert x-y\rvert^2$.

In this paper, we are concerned with a parabolic flow leading to the solution of an optimal transport problem, with cost functions other than the case $c(x, y)=\lvert x-y\rvert^2$. More specifically, we look at solutions to the equation
\begin{equation*}
\begin{cases}{}
\udot(x,t) = \log{\det{(\hess{u(x,t)} - A(x, \gradu(x,t)))}}\\
\qquad\qquad-\log{B(x, \gradu(x,t))},&x\in\Omega\\
\Gbar(x, \gradu(x, t))=0,& x\in\ombdry\\
u\vert_{t=0}=u_0&
\end{cases}
\end{equation*}
with $A$ and $B$ as above, and an appropriate boundary condition $\Gbar$.
It turns out that the conditions given in \cite{MR2512204} along with a few restrictions on the initial condition are sufficient to guarantee long time existence to this parabolic flow, and convergence to the solution of the optimal transport problem as $t\to\infty$.

We also provide here a reference table for the notation used in this paper.

\begin{tabular}{lll}
Notation & Definition & Location\\
\hline\\
$\nu$&Outer unit normal to $\ombdry$&\\
$\nustar$&Outer unit normal to $\omstarbdry$&\\
$\hbar$&Defining function for $\Omega$&\\
$\hstarbar$&Defining function for $\omstar$&\\
$Y(x, p)$&Inverse of the map $y\mapsto\nabla_xc(x, y)=p$&~\eqref{twist}\\
$X(q, y)$&Inverse of the map $x\mapsto\nabla_yc(x, y)=q$&~\eqref{twist}\\
$A(x, p)$&$\hessx{c}(x, y)\vert_{y=Y(x, p)}$&Before~\eqref{MTW}\\
$T(x, t)$&$Y(x, \gradu(x, t))$&~\eqref{notation: TBGbar}\\
$B(x, p)$&$\lvert \det\hessxy{c}\rvert \cdot \frac{f(x)}{g(Y(x, p))}$&~\eqref{notation: TBGbar}\\
$\Btil(x, p)$&$\log{B(x, p)}$&~\eqref{notation: TBGbar}\\
$\Gbar(x, p)$&$\hstarbar(x, Y(x, p))$&~\eqref{notation: TBGbar}\\
$w_{ij}(x, t)$&$u_{ij}(x, t)-A_{ij}(x, \gradu(x, t))$&~\eqref{notation: TBGbar}\\
$\beta(x, t)$&$\Gbar_p(x, p)\vert_{p=\gradu(x, t)}=\hstarbar_l(y)c^{l, k}(x, y)\vert_{y=Y(x, \gradu(x, t))}$&Theorem~\ref{obliquethm}\\
$\omep$&$\Omega \cap B_\epsilon(x_0)$&Before~\eqref{Lvcalc}\\
$M$&$\sup_{x\in\ombdry}\lvert \hess{u(x, t)}\rvert$&Theorem~\ref{intc2est}\\
$\Mtil$&$\sup_{x\in\Omega}\lvert \hess{u(x, t)}\rvert$&Theorem~\ref{ubetabeta}\\
$M_w$&$\sup_{x\in\ombdry}\lvert w_{ij}(x, t)\rvert$&Theorem~\ref{utangent}\\
$\Mtil_w$&$\sup_{x\in\Omega}\lvert w_{ij}(x, t)\rvert$&Theorem~\ref{utangent}\\
$\omepbdry$&$\left\{x\in\Omega \vert \dist(x, \ombdry)<\epsilon\right\}$&Before Theorem~\ref{thm: extended defining functions}\\
$\omstarbdryepstar$&$\left\{y\in\omstar \vert \dist(y, \omstarbdry)<\epstar\right\}$&Before Theorem~\ref{thm: extended defining functions}\\
$h$&Function defined from $\hbar$ with a $c$-convexity property&Theorem~\ref{thm: extended defining functions}\\
$\hstar$&Function defined from $\hstarbar$ with a $c$-convexity property&Theorem~\ref{thm: extended defining functions}\\
$G(x, p)$&Function uniformly convex in $p$&Corollary~\ref{cor: construction of G}\\
\end{tabular}
\\

The author would like to thank Micah Warren for suggesting this problem, and many fruitful discussions regarding this paper.

\section{Preliminaries of Optimal Transport}
In this section, we recall some basic facts and definitions regarding the optimal transport problem, along with the key conditions from \cite{MR2512204}.

Let $\Omega$ and $\omstar$ be open, smooth, bounded domains in $\R^n$, with $f$ and $g$ smooth functions on $\Omega$ and $\omstar$ respectively. We assume the mass balance condition,
\begin{equation}
\int_\Omega f = \int_{\omstar}g
\label{balance}
\end{equation}
along with the bound
\begin{equation}
0<\lambda \leq f, g \leq \Lambda < \infty
\label{densitybound}
\end{equation}
for some constants $\lambda$ and $\Lambda$.
We will assume the following conditions on $c$:
\begin{equation}\label{A0}
c \in C^{4+\alpha}(\omclose \times \omstarclose)\text{ for some }\alpha\in(0, 1].\tag{A0}\\
\end{equation}
We assume the mappings $y \mapsto \nabla_xc(x, y)$ for each $x \in\omclose$ and $x \mapsto \nabla_yc(x, y)$ for each $y \in \omstarclose$ are injective.
For any $p \in \nabla_xc(x, \omstar)$ and $x \in \Omega$, (resp. $q \in \nabla_yc(\Omega, y)$ and $y \in \omstar$) we write $Y(x, p)$ (resp. $X(q, y)$) for the unique element of $\omstar$ (resp. $\Omega$) such that 
\begin{equation}
\begin{aligned}
\nabla_xc(x, y)\vert_{y=Y(x, p)}&=p\\
\nabla_yc(x, y)\vert_{x=X(q, y)}&=q.
\end{aligned}
\tag{A1}
\label{twist}
\end{equation}
We also assume a nondegeneracy condition on the cost $c$:
\begin{equation}
\det\hessxy{c}(x, y) \neq 0, \forall x \in \omclose, y \in \omstarclose. \tag{A2}
\label{nondeg}\\
\end{equation}
Finally, writing $A(x, p) = \hessx{c(x, y)}\vert_{y=Y(x, p)}$ we assume 
\begin{equation}
D_{p_ip_j}A_{kl}(x, p) \xi_i \xi_j \eta_k\eta_l \geq 0, \forall x \in \Omega, p \in \nabla_xc(x, \omstar), \xi \perp \eta.\tag{A3w}
\label{MTW}
\end{equation}

Additionally, we need to make the following assumptions on the domains $\Omega$ and $\omstar$. We write $c_{ij \ldots,\ kl \ldots} = \pdiff{i}\pdiff{j}\ldots\pdiff[y]{k}\pdiff[y]{l}\ldots c$ and indicate the inverse of a matrix by raising its indices.
\begin{defin}\label{def:c-convex domains}
We say that $\Omega$ is $c$-convex with respect to $\omstar$ if the set $\nabla_yc(\Omega, y)$ is a convex set for any $y\in\omstar$. Likewise, $\omstar$ is $c^*$-convex with respect to $\Omega$ if the set $\nabla_xc(x, \omstar)$ is a convex set for any $x\in\Omega$.\\
We say that $\Omega$ is uniformly $c$-convex with respect to $\omstar$ if it is $c$-convex and satisfies
\begin{equation}
[D_i\nu_j(x)-c^{l, k}c_{ij, l}(x, y)\nu_k(x)]\tau_i\tau_j(x) \geq \delbar\text{, }\forall x \in \ombdry,\ y \in \omstar
\label{unifcconvex}
\end{equation}
for some $\delbar > 0$, where $\tau$ is any unit tangent vector to $\ombdry$, and $\nu$ is the outer unit normal to $\ombdry$. Likewise, $\omstar$ is uniformly $c^*$-convex with respect to $\Omega$ if it is $c^*$-convex and satisfies
\begin{equation*}
[D_i\nustar_j(y)-c^{k, l}c_{l, ij}(x, y)\nustar_k(y)]\taustar_i\taustar_j(y) \geq \delstarbar\text{, }\forall y \in \omstarbdry,\ x \in \Omega
\label{unifcstarconvex}
\end{equation*}
for some $\delstarbar > 0$, where $\taustar$ is any unit tangent vector to $\omstarbdry$, and $\nustar$ is the outer unit normal to $\omstarbdry$. 
\end{defin}

\begin{rmk}
For some fixed $y\in\omstar$, given any two points $p_1=\nabla_yc(x_1, y)$ and $p_2=\nabla_yc(x_2, y)$ for $x_1$, $x_2\in\Omega$, we define the $c$-segment with respect to $y$ between $x_1$ and $x_2$ as the inverse image of the straight line between $p_1$ and $p_2$ under the map $\nabla_yc(\cdot, y)$. It is clear that $\Omega$ is $c$-convex with respect to $\omstar$ if and only if every $c$-segment with respect to any $y\in\omstar$ between any two $x_1$ and $x_2\in\Omega$ remains inside $\Omega$.\\
An analogous definition and remark hold for a $c^*$-segment with respect to some $x\in\Omega$ between two points $y_1$, $y_2\in\omstar$.
\end{rmk}

\begin{defin}\label{def:c-convex functions}
For a $y_0\in\omstar$ and a $\lambda_0\in\R$, we call a function of the form $c(\cdot, y_0)+\lambda_0$, a $c$-support function.

We say that a function $\phi$ is $c$-convex, if for every $x_0\in \Omega$ there exists a $c$-support function such that 
\begin{align*}
\phi(x_0)&=c(x_0, y_0)+\lambda_0\\
\phi(x) &\geq c(x, y_0)+\lambda_0\text{, for all }x\neq x_0.
\end{align*}
We say $\phi$ is strictly $c$-convex if the second inequality above is strict.

If $\phi$ is $C^2$, we say that it is locally, uniformly $c$-convex if
\begin{equation*}
\phi_{ij}(x)-A_{ij}(x, Y(x, \gradphi(x)))>0 
\end{equation*}
as a matrix, for every $x\in\omclose$.
\end{defin}

\section{The Main Problem}
The symbols $\nabla$, $\hess{}$, and $D_i$ will denote differentiation in the $x$ variables, with specific variables indicated by subscripts when necessary. The notation $\udot$ will indicate derivatives in the $t$ direction. Now, by the smoothness of $\omstarbdry$, we may extend $g$ to all of $\R^n$ so that it is $C^2$ and satisfies the bound $\frac{\lambda}{2}\leq g \leq 2\Lambda$. Writing 
\begin{align}\label{notation: TBGbar}
T(x, t) &= Y(x, \gradu(x, t))\notag\\
B(x, p)&= \lvert\det \hessxy{c(x, y)}\rvert_{y=Y(x, p)}\cdot\frac{f(x)}{g(Y(x, p))}\notag\\
\Btil(x, p) &= \log{B(x, p)}\notag\\
\Gbar(x, p)&=\hstarbar(Y(x, p))\notag\\
w_{ij}(x, t) &= u_{ij}(x, t) - A_{ij}(x, \gradu(x, t)).
\end{align}
where $\hstarbar$ is a normalized defining function for $\omstar$ (meaning $\nabla\hstarbar= \nustar$ on $\omstarbdry$, $\hstarbar=0$ on $\omstarbdry$, and $\hstarbar < 0$ on $\omstar$), we consider solutions $u=u(x, t)$ to the following equation:

\begin{numcases}{\label{eqn}}
\udot(x,t) = \log{\det{(\hess{u(x,t)} - A(x, \gradu(x,t)))}}\notag\\
\qquad\qquad-\log{B(x, \gradu(x,t))},&$x\in\Omega$\label{floweqn}\\
\Gbar(x, \gradu(x, t))=0,& $x\in\ombdry$\label{bdrycond}\\
u\vert_{t=0}=u_0.&
\end{numcases}
We require the following conditions on $u_0\in C^{2+\alpha}(\Omega \times \{0\})$:
\begin{numcases}{\label{u_0 conditions}}
u_0\text{ is locally, uniformly }c\text{-convex}\\
\hstarbar(Y(x, \gradu_0(x)))=0\text{ on }\ombdry\\
T_0(\Omega) = \omstar
\end{numcases}
where $T_0(x)=Y(x, \gradu_0(x))$.

The main theorem we prove is the following.
\begin{thm}\label{thm: main theorem}
Suppose that $c$ satisfies conditions~\eqref{A0}-~\eqref{MTW}. Additionally, suppose that $\Omega$ and $\omstar$ are uniformly $c$ and $c^*$-convex with respect to each other.

If $u_0\in C^{2+\alpha}(\Omega \times \{0\})$ satisfies the conditions~\eqref{u_0 conditions}, there exists a solution $u$ to equation~\eqref{eqn} for all times $t\geq 0$ which is $C^2(\omclose)$ in the $x$ variables and $C^1(\R_+)$ in the $t$ variable.\\
In addition, $u(\cdot, t)$ converges in $C^2(\Omega)$ to a $c$-convex function $u^\infty(\cdot)$ as $t\to \infty$ which satisfies the elliptic optimal transport equation:
\begin{numcases}{\label{elliptic}}
\det{(\hess{u^\infty(x)} - A(x, \gradu^\infty(x)))}=B(x, \gradu^\infty(x)),&$x\in\Omega$\label{elliptic eqn}\\
Y(\cdot, \gradu^\infty(\cdot))[\Omega]=\omstar.&\label{elliptic bdrycond}
\end{numcases}
\end{thm}

\begin{rmk}\label{rmk:Tu and c-support functions}
Note that if $\phi(\cdot)$ is differentiable at some $x_0\in\Omega$, and $c(x, y_0)+\lambda_0$ is a $c$-support function to $\phi$ at $x_0$, we will have that $\nabla\phi(x_0)=\nabla_xc(x_0, y_0)$. Thus, by the uniqueness in~\eqref{twist}, we have that $y_0=Y(x_0, \nabla\phi(x_0))$. 
\end{rmk}

\section{Short-time Existence}
We will prove the existence of a solution to our equation~\eqref{eqn} up to some small time $\tmax>0$. First, we will prove some auxiliary results. We follow the definitions for H\"older spaces given in \cite{Lieberman:1996:SOP}.
\begin{defin}
Write $X=(x, t)$, $X_0=(x_0, t_0)\in\Omega \times I$ for some interval $I$, and $\lVert X\rVert =\max{(\lvert x\rvert, \lvert t\rvert^{\frac{1}{2}}})$. For a function $f$ defined on $\Omega \times I$, with $k$ a positive integer and $0<\alpha\leq 1$, we define the following norms and seminorms:
\begin{align*}
[f]_{C^{k+\alpha}(\Omega \times I)}&=\sum_{\lvert\beta\rvert+2j=k}\sup_{X\neq X_0\in\Omega \times I}{\frac{\lvert D^\beta_xD^j_t(f(X)-f(X_0))\rvert}{\lVert X-X_0\rVert^\alpha}},\\
\langle f\rangle_{C^{k+\alpha}(\Omega \times I)}&=\sum_{\lvert\beta\rvert+2j=k-1}\sup_{X_0,\ t\neq t_0\in I}{\frac{\lvert D^\beta_xD^j_t(f(x_0, t)-f(X_0))\rvert}{\lvert t-t_0\rvert^\frac{\alpha}{2}}},\\
\lvert f\rvert_{C^{k+\alpha}(\Omega \times I)}&=[f]_{k+\alpha}+\langle f\rangle_{k+\alpha}\\
\lVert f\rVert_{C^{k+\alpha}(\Omega \times I)}&=\sum_{\lvert \beta\rvert+2j\leq k}\sup\lvert D^\beta_xD^j_t f\rvert +\vert f\rvert_{k+\alpha}.
\end{align*}
We write $f\in C^{k+\alpha}(\Omega \times I)$ if $\lVert f\rVert_{C^{k+\alpha}(\Omega \times I)}<\infty$.
\end{defin} 
\begin{rmk}\label{rmk: regularity to the boundary}
By looking at sequences of points in $\Omega \times I$ approaching $\ombdry \times I$ or $\Omega \times \{0\}$, it is easy to see that if $f\in C^{k+\alpha}$, the $x$ derivatives up to order $k$ and the $t$ derivatives up to order $\frac{k}{2}$ of $f$ exist and are continuous up to $(\ombdry \times I) \cup (\Omega \times \{0\})$.
\end{rmk}

\begin{defin}
\begin{equation*}
\mathcal{K}=\left\{\phi\in C^{2+\alpha}(\Omega \times I)\ \vert\ \phi(\cdot, t)\text{ is locally, uniformly }c\text{-convex for each }t\in I\right\}.
\end{equation*}
\end{defin}
For this section, let us write 
\begin{equation*}
F(x, p, r)=\log{\det{(r_{ij} - A(x, p))}}-\log{B(x, p)}
\end{equation*}
well defined for $(x, p, r)$ in some subset of $\Omega \times \R^n \times \R^{n^2}$ (in particular, $F$ is well defined when $r=\hess\phi(x, t)$, $p=\gradphi(x, t)$ for $\phi \in \mathcal{K}$.)

By Theorem~\ref{obliquethm} below, the boundary condition~\eqref{bdrycond} is a nonlinear oblique condition. By a modification of the argument in \cite[Theorem 2.5.7]{Gerhardt:CurvatureProblems}, we can show short-time existence for our equation~\eqref{eqn}. The main difference is that the nonlinear boundary condition makes the set of admissible solutions not a Banach space, but a Banach manifold.\\

\begin{thm}\label{thm: short time existence}
For some $\tmax>0$, we can find a solution to~\eqref{eqn} in the class $C^{4, 2}(\omclose\times[0, \tmax))$. This notation means the solution $u$ is $C^4$ in the $x$ variables, and $C^2$ in the $t$ variable.
\end{thm}
\begin{proof}
Consider $\util$, the solution to the problem
\begin{equation*}
\begin{cases}{\label{shorttime eqn}}
\Delta\util-\utildot = \Delta u_0-F(x, \gradu_0, \hess{u_0}), &x\in\Omega\\
\Gbar(x, \nabla\util(x, t))=0,& x\in\ombdry\\
\util\vert_{t=0}=u_0.
\end{cases}
\end{equation*}
By Theorem~\ref{obliquethm}, we have $\inner{\Gbar_p}{\nu}>0$, thus on $[0, \epsilon)$ for some small $\epsilon$, we find 
\begin{equation*}
\inner{\Gbar_p}{\nu}\geq C_\epsilon>0.
\end{equation*}
Also, by assumption $\Gbar(x, \gradu_0(x))=0$ for $x\in\ombdry$, so by \cite[Theorems 8.8 and 8.9]{Lieberman:1996:SOP} there exists a solution $\util \in C^{2+\alpha}(\Omega \times [0, \epsilon))$ (for a possibly smaller $\epsilon >0$). In particular, $\lVert \util(\cdot, t)-u_0(\cdot)\rVert_{C^2}<C\lvert t\rvert^\alpha$ for each $t$. Thus, since $u_0$ is locally, uniformly $c$-convex by assumption, (by making $\epsilon$ smaller if necessary) we can ensure that $\util \in\mathcal{K}$. Writing $\wtil^{ij}$ for the matrix inverse of $\util_{ij}-A_{ij}(x, \nabla \util)$ we see that $\wtil^{ij}$ is positive definite for $t<\epsilon$ and $\wtil^{ij}\xi_i\xi_j\geq C_ \epsilon\lvert\xi\rvert^2$ for some $C_\epsilon>0$. Define
\begin{equation*}
\ftil(x, t):=F(x, \nabla\util, \hess{\util})-\utildot(x, t).
\end{equation*}
Note that $\ftil\in C^\alpha(\Omega \times [0, \epsilon))$ and $\ftil\vert_{t=0}=0$.

We now define the following sets for $0< \delta < \frac{\epsilon}{2}$: 
\begin{equation*}
\Bdel:=\{\phi\in C^{2+\frac{\alpha}{4}}(\Omega \times [0, \epsilon))\ \lvert\ \lVert \phi-\util\rVert_{C^{2+\frac{\alpha}{4}}(\Omega \times [0, \epsilon))}<\delta\}
\end{equation*}
and
\begin{equation*}
\Bdel^*:=C^{\frac{\alpha}{4}}(\Omega \times [0, \epsilon))\times \mathcal{B}
\end{equation*}
where
\begin{equation*}
\mathcal{B}:=\{(v, w)\in C^{2+\frac{\alpha}{4}}(\Omega \times \{0\})\times C^{1+\frac{\alpha}{4}}(\ombdry \times [0, \epsilon))\ \vert\ \Gbar(x, \nabla v(x))=w(x, 0),\ \forall x\in\ombdry\}
\end{equation*}
Clearly, $\Bdel$ is a $C^1$ Banach manifold with charts diffeomorphic to subsets of $C^{2+\frac{\alpha}{4}}(\Omega \times [0, \epsilon))$.

We also claim that $\mathcal{B}$ is a Banach manifold. Consider the map $H: C^{2+\frac{\alpha}{4}}(\Omega \times \{0\})\times C^{1+\frac{\alpha}{4}}(\ombdry \times [0, \epsilon)) \to C^{1+\frac{\alpha}{4}}(\ombdry \times \{0\})$ given by
\begin{equation*}
H(v, w) := \Gbar(x, \nabla v(x))-w(x, 0).
\end{equation*}
We see that the differential
\begin{equation*}
D_{(v_0, w_0)}H(v, w) = \inner{\Gbar_p(x, \nabla v_0(x))}{\nabla v}-w(x, 0)
\end{equation*}
is clearly onto $C^{1+\frac{\alpha}{4}}(\ombdry \times \{0\})$ for each $(v_0, w_0)\in\mathcal{B}$. Also, we consider the map $P: C^{2+\frac{\alpha}{4}}(\Omega \times \{0\})\times C^{1+\frac{\alpha}{4}}(\ombdry \times [0, \epsilon))\to \mathcal{N}(D_{(v_0, w_0)}H(v, w))$, the kernel of $D_{(v_0, w_0)}H(v, w)$, defined by
\begin{equation*}
P(v, w):=(v, w+\phi_{v, w})
\end{equation*}
where 
\begin{equation*}
\phi_{v, w}(x, t)=\inner{\Gbar_p(x, \nabla v_0(x))}{\nabla v(x, 0)}-w(x, 0)\text{, for all }(x, t)\in \ombdry \times (0, \epsilon).
\end{equation*}
Since $P$ is a continuous linear projection, $\mathcal{N}(D_{(v_0, w_0)}H(v, w))$ splits the full tangent space of $C^{2+\frac{\alpha}{4}}(\Omega \times \{0\})\times C^{1+\frac{\alpha}{4}}(\ombdry \times [0, \epsilon))$ at $(v_0, w_0)$ and we have that $(v_0, w_0)$ is a regular point of $H$. Thus, $\mathcal{B}$ is the inverse image of a regular value of $H$, so by \cite[Theorem 73C]{Zeidler:1988:NonlinearFunctionalV4} is also a Banach manifold with tangent space at $(v_0, w_0)$ equal to $\mathcal{N}(D_{(v_0, w_0)}H(v, w))$. Hence $\Bdel^*$ is also a $C^1$ Banach manifold.

We now define the map 
\begin{equation*}
\Phi: \Bdel \to \Bdel^*
\end{equation*}
by
\begin{equation*}
\Phi(u)=(F(x, \gradu, \hess{u})-\udot,\ u\vert_{t=0},\ \Gbar(x, \gradu)).
\end{equation*}
By the uniform continuity of $c_{ij}$, we can see that for a function $v\in C^2(\Omega)$, if $\hess v-A(x, Y(x, \nabla v))$ is positive definite, it remains so on a small neighborhood of $v$ in the $C^2(\Omega)$ norm. Hence, if $\delta$ is sufficiently small, any convex combination of $\util$ and $\phi\in \Bdel$ will remain in $\mathcal{K}$. Thus we find that the differential of $\Phi$ at $\util$ is given by
\begin{align*}
D_{\util}\Phi(\phi)&=(\wtil^{ij}(\phi_{ij}-\Dp A_{ij}(x, \nabla\util)\phi_k)-\Dp \Btil(x, \nabla \util)\phi_k-\phidot, \\
&\qquad\phi\vert_{t=0},\ \inner{\Gbar_p(x, \nabla \util)}{\nabla\phi}).
\end{align*}

Since by above, the tangent space to $\Bdel^*$ at $\Phi(\util)$ is $C^{\frac{\alpha}{4}}(\Omega \times [0, \epsilon))\times \mathcal{N}(D_{(\Phi(\util))}H)$, we can use \cite[Theorem 5.18]{Lieberman:1996:SOP} to find that $D_{\util}\Phi$ is bijective on its tangent space (the condition that a pair be in the kernel of the differential of $H$ is exactly the required compatibility condition between the initial and boundary conditions to obtain short time existence for a linear parabolic equation). Hence, by the inverse function theorem for Banach manifolds, $\Phi$ is invertible on some small neighborhood of $\util\in\Bdel$.

Now, taking a smooth cutoff function $\eta$ in the variable $t$ such that 
\begin{equation*}
\eta_\delta(t)=
\begin{cases}
0,&\ 0\leq t<\delta\\
1,&\ 2\delta\leq t
\end{cases}
\end{equation*}
with $\lvert \etadot_\delta\rvert \leq \frac{C}{\delta}$, consider the function $\ftil_\delta(x, t) :=\eta_\delta(t)\ftil(x, t)$. Since $\ftil\vert_{t=0}=0$, we can make $\displaystyle \sup_{X\in\Omega \times [0, \epsilon)}{\vert \ftil_\delta-\ftil\rvert}$ small by taking $\delta$ small. On the other hand, we calculate that
\begin{align*}
\lvert \ftil_\delta-\ftil\rvert_{C^\frac{\alpha}{4}(\Omega \times [0, \epsilon))}&\leq \lvert \ftil\rvert_{C^\frac{\alpha}{4}(\Omega \times [0, \delta))}+ C\lvert \eta_\delta-1\rvert_{C^\frac{\alpha}{4}(\Omega \times [\delta, 2\delta))}\lvert \ftil\rvert_{C^\frac{\alpha}{2}(\Omega \times [\delta, 2\delta))}
\end{align*}
where $C$ is a constant depending only on $\Omega$ and $\lVert\ftil\rVert_{C^{\alpha}(\Omega \times [0, \epsilon))}$. Since
\begin{equation*}
\lvert \ftil(x, t)\rvert =\lvert \ftil(x, t)-\ftil(x, 0)\rvert \leq Ct^\alpha,
\end{equation*} 
we have
\begin{align*}
\lvert \ftil\rvert_{C^\frac{\alpha}{4}(\Omega \times [0, \delta))}&\leq\sup\frac{\lvert\ftil(x, t)-\ftil(x_0, t_0))\rvert}{\lvert x-x_0\rvert^{\frac{\alpha}{4}}}+\sup\frac{\lvert\ftil(x, t)-\ftil(x_0, t_0))\rvert}{\lvert t-t_0\rvert^{\frac{\alpha}{8}}}\\
&\leq\sup\left(\frac{\lvert\ftil(x, t)-\ftil(x_0, t_0))\rvert}{\lvert x-x_0\rvert^\alpha}\lvert\ftil(x, t)-\ftil(x_0, t_0))\rvert^3\right)^\frac{1}{4}\\
&\qquad+\sup\frac{\lvert\ftil(x, t)-\ftil(x_0, t_0))\rvert}{\lvert t-t_0\rvert^{\frac{\alpha}{2}}}\lvert t-t_0\rvert^{\frac{3\alpha}{8}}\\
&\leq C(\delta^{\frac{3\alpha}{4}}+\delta^{\frac{3\alpha}{8}}),
\end{align*}
while by a similar calculation,
\begin{align*}
\lvert \eta_\delta-1\rvert_{C^\frac{\alpha}{4}(\Omega \times [\delta, 2\delta))}\lvert \ftil\rvert_{C^\frac{\alpha}{2}(\Omega \times [\delta, 2\delta))}&
\leq C\left(\sup\frac{\lvert\eta_\delta(t)-\eta_\delta(t_0))\rvert}{\lvert t-t_0\rvert^{\frac{\alpha}{8}}}\right)(\delta^{\frac{\alpha}{2}}+\delta^{\frac{\alpha}{4}})\\
&\leq C\left(\sup\frac{\lvert\etadot\rvert(t-t_0)}{\lvert t-t_0\rvert^{\frac{\alpha}{8}}}\right)(\delta^{\frac{\alpha}{2}}+\delta^{\frac{\alpha}{4}})\\
&\leq C\delta^{-1}\delta^{1-\frac{\alpha}{8}}(\delta^{\frac{\alpha}{2}}+\delta^{\frac{\alpha}{4}})
\end{align*}
Thus for a small enough $\delta>0$, $\eta_\delta\ftil$ is arbitrarily close to $\ftil$ in $C^{\frac{\alpha}{4}}(\Omega \times [0, \epsilon))$. This gives a $u\in C^{2+\frac{\alpha}{4}}(\Omega \times [0, \epsilon))$ such that
\begin{equation*}
\begin{cases}
F(x, \gradu, \hess{u})=\eta_\delta\ftil, &x\in\Omega\\
\Gbar(x, \gradu(x, t))=0, &x\in\ombdry\\
u\vert_{t=0}=u_0.&
\end{cases}
\end{equation*}
Hence, by taking $\tmax < \delta$, we see that $u$ is the solution to our original problem~\eqref{eqn} on $[0, \tmax)$ in the space $C^{2+\frac{\alpha}{4}}(\Omega \times [0, \tmax))$. Thus by Remark~\eqref{rmk: regularity to the boundary}, we have the desired regularity.
\end{proof}

From here on, we assume that $u$ exists with this regularity for $t\in[0, \tmax)$.

\section{Preliminary results}
We will show a few preliminary results before proceeding to show the estimates necessary for long term existence.\\

\begin{rmk}\label{rmk: DT is nondegenerate}
By implicit differentiation, we see that $D_iT^j(x, t)=c^{j, k}w_{ki}(x, T(x, t))$ and thus $\det{DT}=\frac{f}{g\circ T}e^{\udot}$ from~\eqref{floweqn}. In particular, $\det{DT}\neq0$ for all $x\in\omclose$ and $t\in[0, \tmax)$. We will use this fact frequently.
\end{rmk}
\begin{rmk}\label{rmk: formulae for inverse of c_x,y}
We will also make use of the following formulas obtained by a simple differentiation:
\begin{align*}
c^{l, k}_i&=-c_{ip, q}c^{l, p}c^{q, k}\\
c^{l, k}_{\ ,i}&=-c_{p, iq}c^{l, p}c^{q, k}
\end{align*}  
\end{rmk}

\begin{lemma}\label{lemma: stays strictly c-convex}
If $u_0$ satisfies the conditions~\eqref{u_0 conditions}, the solution $u(\cdot, t)$ to~\eqref{eqn} is both locally, uniformly $c$-convex and strictly $c$-convex for $0\leq t<\tmax$. In particular, $w_{ij}$ remains positive definite for $0\leq t<\tmax$.
\end{lemma}

\begin{proof}
By the assumption on the initial condition $u_0$, we see that $w_{ij}$ remains positive definite for at least some small time. Then, since $\det{w_{ij}}=Be^{\udot}$, by conditions~\eqref{nondeg} and~\eqref{densitybound} we see that $w_{ij}$ cannot have $0$ for an eigenvalue, i.e. it will remain positive definite as long as the solution exists, proving the claim of local, uniform $c$-convexity.

Now, suppose that for some $t$, $u(\cdot, t)$ is $c$-convex but not strictly $c$-convex, i.e. there exists $x_1\neq x_2$ such that $c(\cdot, \ytil)+\lambdatil$, for some $\ytil$ and some $\lambdatil\in\R$ is a $c$-support function to $u(\cdot, t)$ at both $x_1$ and $x_2$.  Define $\vtil(x, t)=u(x, t) -(c(x, \ytil)+ \lambdatil)$. By local, uniform $c$-convexity we see that $\hess\vtil(x_1)=(w_{ij}(x_1))>0$, while by the definition of $T$ and Remark~\ref{rmk:Tu and c-support functions}, we have that $\nabla \vtil =0$ at $x=x_1$. This implies that $\vtil(x_1, t)=0$ is a strict local minimum on some neighborhood of $x_1$. By the continuity of the map $X(q, y)$ in~\eqref{twist}, we can pick a point $x_0$ on the $c$-segment with respect to $\ytil$ between $x_1$ and $x_2$ that lies in this neighborhood of $x_1$. The $c$-convexity of $\Omega$ with respect to $\omstar$ ensures that $x_0\in\Omega$. Since $u(\cdot, t)$ is $c$-convex, there is a $y_0\in\omstar$ and a $\lambda_0\in\R$ so that $c(x, y_0)+\lambda_0$ is a $c$-support function at $x_0$. 
In particular, 
\begin{equation*}
c(x, \ytil)+\lambdatil =u(x) \geq c(x, y_0)+\lambda_0\text{ at }x=x_1,\ x_2.
\end{equation*}
Since we assume condition~\eqref{MTW}, we can use  \cite[Theorem 3.2] {MR2506751}, to show that $\lambda_0-\lambdatil\leq\phibar(\ytil)\leq c(x_0, \ytil)-c(x_0, y_0)$, where
\begin{equation*}
\phibar(y):=\min{\{c(x_1, y)-c(x_1, y_0), c(x_2, y)-c(x_2, y_0)\}}
\end{equation*}
(recalling that $x_0$ lies on the $c$-segment with respect to $\ytil$ between $x_1$ and $x_2$). We make note that, due to the fact that the sign of the potential function in \cite {MR2506751} is the negative of ours, and the $x$ and $y$ variables are reversed in the theorem, the statement becomes: 
\begin{equation*}
\phibar(y) \leq c(x_\theta, y)-c(x_\theta, y_0)
\end{equation*}
for any $x_\theta\in\Omega$ lying on the $c$-segment with respect to $y_0$ connecting $x_1$ and $x_2$ and any $y\in\omstar$. But then, $u(x_0, t)=c(x_0, y_0)+\lambda_0\leq c(x_0, \ytil)+\lambdatil$, contradicting that $x_1$ is a strict local minimum for $\vtil$. Thus $u(\cdot, t)$ is actually strictly $c$-convex.

Now consider the function 
\begin{equation*}
v(x,x_0, t)=u(x, t)-(c(x, T(x_0, t))+u(x_0)-c(x_0, T(x_0, t)))
\end{equation*}
defined on $\Omega \times \Omega \times [0, \tmax)$. Take $t\leq\tbar < \tmax$, then $(w_{ij}(x, t))\geq C\delta_{ij}$ as a matrix, for some $C$ depending on $\tbar$ but independent of $x$ and $t$. Then, since $w_{ij}(x, t)=\hessx{v(x, x_0, t)}$ at $x=x_0$, we have that
\begin{align*}
v(x, x_0, t)&\geq v(x_0, x_0, t)+\inner{\nabla_x v(x_0, x_0, t)}{x-x_0}\\
&\qquad+\inner{\hessx v(x_0, x_0, t)(x-x_0)}{x-x_0} - \lVert v \rVert_{C^3}\lvert x-x_0\rvert^3\\
&\geq 0+0+C\lvert x-x_0\rvert^2- \lVert v \rVert_{C^3}\lvert x-x_0\rvert^3\\
&> 0
\end{align*}
as long as $0<\lvert x-x_0\rvert < \delta$, for $\delta$ depending only on $C$ and 
\begin{equation*}
\lVert v \rVert_{C^3}=\sup_{i,j,k}\sup_{(x, t)\in \omclose\times[0, \tbar]} \lvert \partial^3_{x_ix_jx_k}v(x_0, x, t)\rvert.
\end{equation*}
Now, if $\lvert x-x_0\rvert \geq \delta$, by the strict $c$-convexity of $u_0$ we have
\begin{align*}
v(x, x_0, 0)&=u_0(x)-c(x, T_0(x_0))-u_0(x_0)+c(x_0, T_0(x_0))\\
&\geq K(x_0)>0\\
\end{align*}
where $\displaystyle K(x_0) := \inf_{x}{v(x, x_0, 0)}$. By the continuity of $u_0$, $T_0$, and $c$ we have that $K(x_0)$ is actually a continuous function of $x_0$. Thus by boundedness of $\Omega$, we have that $\displaystyle \inf_{x_0\in\omclose}K(x_0) \geq K_0$ for some constant $K_0>0$. Now by the continuity of $v$, for some sufficiently small $\epsilon >0$, we have that $\lvert v(x, x_0, t)-v(x, x_0, 0)\rvert < \frac{K_0}{2}$. Thus, for $0\leq t< \epsilon$, we have
\begin{align*}
v(x, x_0, t)&>0\text{, for }0<\lvert x-x_0\rvert< \delta\\
v(x, x_0, t)&\geq v(x, x_0, 0)-\lvert v(x, x_0, t)-v(x, x_0, 0)\rvert \\
&\geq K_0-\frac{K_0}{2}=\frac{K_0}{2}>0\text{, for }\lvert x-x_0\rvert\geq \delta.
\end{align*}
We can see this implies that $u(\cdot, t)$ remains strictly $c$-convex for $0\leq t < \epsilon$.\\
Taking $\ttil=\sup\{t<\tmax \vert u(\cdot, \tbar)\text{ is strictly } c\text{-convex for } \tbar<t\}$, by the argument above, $\ttil > 0$. If $\ttil<\tmax$ we have that $u(\cdot, \ttil)$ is $c$-convex but not strictly $c$-convex, which is a contradiction. Thus we have that $\ttil =\tmax$ and $u(\cdot, t)$ remains strictly $c$-convex on $0\leq t<\tmax$.\\
\end{proof}
\begin{cor}\label{optimalonetoone}
Under the assumptions of Lemma~\ref{lemma: stays strictly c-convex} above, the map $T(x, t)=Y(x, \gradu(x, t))$ is one-to-one on $\omclose$ for each $0\leq t<\tmax$.
\end{cor}
\begin{proof}
Suppose that $\ytil=T(x_1, t)=T(x_2, t)$ for some $x_1\neq x_2$. By Remark~\ref{rmk:Tu and c-support functions}, we have that $c(x, \ytil)+\lambda_1$ and $c(x, \ytil)+\lambda_2$ are $c$-support functions to $u(\cdot, t)$ at $x_1$ and $x_2$ respectively, for some $\lambda_1$, $\lambda_2\in\R$. Thus
\begin{align*}
c(x_2, \ytil)+\lambda_1&\leq u(x_2, t)=c(x_2, \ytil)+\lambda_2\\
c(x_1, \ytil)+\lambda_2&\leq u(x_1, t)=c(x_1, \ytil)+\lambda_1
\end{align*}
and hence $\lambda_1=\lambda_2$. This contradicts the strict $c$-convexity of $u(\cdot, t)$, and we have that $T(\cdot, t)$ is a one-to-one function for each $t<\tmax$.
\end{proof}

\begin{cor}\label{cor: T is onto}
Under the assumptions of Lemma~\ref{lemma: stays strictly c-convex}, $T(\Omega, t)=\omstar$ for $0\leq t<\tmax$.
\end{cor}
\begin{proof}
By $c$-convexity and the conditions~\eqref{twist} and~\eqref{nondeg}, $\Omega$ and $\omstar$ are homeomorphic to convex sets, and thus the unit ball in $\R^n$, and also their boundaries are connected. Since the argument here will be entirely topological, we may compose $T$ with the appropriate homeomorphisms, and assume $\Omega=\omstar=B_1(0)$. The boundary condition~\eqref{bdrycond} implies that $T(\ombdry, t)\subset \omstarbdry$ for each $t$. $T$ is continuous on $\ombdry$, and by Corollary~\ref{optimalonetoone}, is one-to-one there. By Remark~\ref{rmk: DT is nondegenerate}, $\Tinv$ is also continuous on $T(\ombdry, t)$ and thus $T(\cdot, t)$ is a homeomorphism on $\ombdry$. Hence, 
\begin{equation}\label{eq: image of boundary}
T(\ombdry, t)=\omstarbdry\text{ for each }t.
\end{equation}
Suppose there is some $t_0<\tmax$ such that $T(\Omega, t_0)\not\subset\omstar$. That means, for some $x_0\in\Omega$, $T(x_0, t_0)\not\in\omstar$. However, since $T(x_0, 0)\in\omstar$ by assumption, by using the intermediate value property on the continuous, real valued function $\phi(t)=\lvert T(x_0, t)\rvert$, there must be a $t\in(0, t_0)$ such that $T(x_0, t)\in\omstarbdry$, contradicting that $T(\cdot, t)$ maps $\ombdry$ one-to-one and onto $\omstarbdry$. Thus, $T(\Omega, t)\subset\omstar$ and hence $T(\omclose, t)\subset\omstarclose$ for any $t<\tmax$.\\
As above, $T(\cdot, t)$ is a homeomorphism on $\omclose$ as well, so we actually have $T(\omclose, t)=\omstarclose$, $t<\tmax$. But since $T(\ombdry, t)=\omstarbdry$, and since $T(\cdot, t)$ is one-to-one on $\omclose$, we see that
\begin{equation*}
T(\Omega, t)=\omstar, t<\tmax.
\end{equation*}
\end{proof}

\begin{rmk}
By the above Corollary, we can now see that the solution $u$ will be independent of the extension that we chose for $g$ outside of $\omstar$.
\end{rmk}

\begin{lemma}\label{reverseproblem}
Under the conditions of Corollary~\ref{optimalonetoone}, define the $c$-transform of $u(x, t)$ as $u^*(y, t) = c(x, y)-u(x, t)$, where $y=T(x, t)$. Also, define $T^*(y, t) = X(\gradustar(y, t), y)$. Then, 

\begin{enumerate}
\item $T^*(\cdot, t)=T^{-1}(\cdot, t)$.
\item $u^*$ satisfies the following flow equation
\begin{equation*}
\begin{cases}{\label{reveqn}}
\ustardot(y,t) = \log{\det{(\nabla^2u^*(y,t) - A^*(\gradustar(y,t), y))}}\\
\qquad\qquad-\log{B^*(\gradustar(y,t), y)},\qquad y\in\omstar&\\
\hbar(X(\gradustar(y,t), y))=0,\qquad\qquad\quad y\in\omstarbdry&\\
u^*\vert_{t=0}=(u_0)^*&
\end{cases}
\end{equation*}
for $0\leq t<\tmax$ where 
\begin{align*}
(u_0)^*(T_0(x))&=c(x, T_0(x))-u(x)\\
A^*(y, q) &= \hessy c(X(q, y), y)\\
B^*(y, q)&=\lvert\det \hessxy{c(x,y)}\rvert_{x=X(q, y)}\cdot\frac{g(y)}{f(X(q, y))}
\end{align*}
for $f$ extended to be in $C^2(\R^n)$ with bounds $\frac{\lambda}{2}\leq f\leq 2\Lambda$, and $\hbar$ a defining function for $\Omega$.
\end{enumerate}
\end{lemma}

\begin{proof}
Since we know that $T(\cdot, t)$ is invertible for each $t$, we may fix a $t$ and write $u^*(y, t) = c(\Tinv(y, t), y)-u(\Tinv(y, t), t)$ and $x=\Tinv(y, t)$. We will drop all $t$ for ease of notation here. Differentiating both sides in $y$ and remembering the definition of $T$ gives us
\begin{align*}
\gradustar(y) &=D(\Tinv)\cdot [\nabla_xc(\Tinv(y), y)-\gradu(\Tinv(y))]+\nabla_yc(\Tinv(y), y)\\
&=D(\Tinv)\cdot [\nabla_xc(x, T(x))-\gradu(x)]+\nabla_yc(\Tinv(y), y)\\
&=\nabla_yc(\Tinv(y), y).
\end{align*}
Thus by the uniqueness in assumption~\eqref{twist}, we have that $T^*(y, t) = \Tinv(y, t)$.\\
Next, we differentiate both sides of the relation$\gradustar(y)=\nabla_yc(\Tinv(y), y)$ to obtain
\begin{align*}
\hessy u^*&=\hessy c(\Tinv(y), y)+D(\Tinv)(y)\cdot\hessxy c(\Tinv(y), y)\\
&=\hessy c(T^*(y), y)+[D(T)]^{-1}\vert_{x=T^*(y)}\cdot\hessxy c(T^*(y), y).\\
\end{align*}
Rearranging and taking determinants of both sides, and using Remark~\ref{rmk: DT is nondegenerate}, we obtain
\begin{equation*}
\det{(\hessy u^*-\hessy c(T^*(y), y))}=e^{-\udot(T^*(y))}\frac{g(y)}{f(T^*(y))}\cdot\lvert\hessxy c(T^*(y), y)\rvert.
\end{equation*}
Now, differentiating $u^*(y, t) = c(\Tinv(y, t), y)-u(\Tinv(y, t), t)$ in $t$, we have
\begin{align*}
\ustardot(y, t)&=D(\Tinv(y, t))\cdot [\nabla_xc(\Tinv(y, t), y)-\gradu(\Tinv(y, t), t)]\\
&\qquad-\udot(\Tinv(y, t), t)\\
&=-\udot(T^*(y, t), t).
\end{align*}
Combining the above, we obtain the desired equation for $u^*$.
\end{proof}

\section{Estimate of $\nabla u$}
\begin{thm}\label{gradientbdd}
As long as a solution to the equation~\eqref{eqn} exists on a time interval $[0, \tmax)$, $\lvert \nabla u\rvert \leq C$ for some $C>0$ depending on $\omstar$ and $c$, but independent of $\tmax$.
\end{thm}
\begin{proof}
By Corollary~\ref{cor: T is onto} and the boundedness of $\omstar$, $\lvert T\rvert$ is bounded independent of $t$. From the relationship $\nabla_xc(x,y)\rvert_{y=T(x)}=\gradu(x, t)$ and the boundedness of $\lvert\nabla_xc\rvert$ the estimate is immediate.
\end{proof}

\section{Obliqueness of the Boundary Condition}
Let $\nu$ and $\nustar$ be the outward unit normal vectors to $\Omega$ and $\omstar$ respectively. Writing $\Gbar(x, p) = \hstarbar(Y(x, p))$, the boundary condition~\eqref{bdrycond} can be written
\begin{equation*}
\Gbar(x, \gradu) =0\text{ for }x\in\ombdry.
\end{equation*}

\begin{thm}\label{obliquethm}
As long as the solution to~\eqref{eqn} exists, we have obliqueness of the boundary condition, i.e.
\begin{equation}
\inner{\Gbar_p(x, \gradu)}{\nu(x)} >0, x\in \ombdry, t>0.
\label{obliqueness}
\end{equation}
\end{thm}

\begin{proof}
First we note that $\hstarbar(T(x, t))=0$ for $x \in \ombdry$, so for any $\tau$ which is tangential to $\ombdry$, we find that $\hstarbar_l D_iT^l\tau_i=0$. Also, since $\hstarbar <0$ in $\Omega$, we find from this 
\begin{equation}
\hstarbar_l D_iT^l=\hstarbar_l c^{l, k}w_{ki}=\chi \nu_i 
\label{urbas0}
\end{equation}
or equivalently, $\hstarbar_l c^{l, i}=\chi w^{ik}\nu_k$for some $\chi \geq 0$. Writing $\beta(x, t) = \nabla_p \Gbar(x, p)\vert_{p=\gradu}=\Gbar_p(x, \gradu)$, we see that $\beta_k = \hstarbar_lY^l_{p_k} \vert_{p=\gradu} = \hstarbar_lc^{l, k}\vert_{p= \gradu}$, hence 

\begin{equation}
\inner{\beta}{\nu}=\chi w^{kl}\nu_k\nu_l.
\label{urbas1}
\end{equation}

By Remark~\ref{rmk: DT is nondegenerate}, $\det{DT} \neq 0$, and hence $\chi >0$. By Lemma~\ref{lemma: stays strictly c-convex}, $w^{ij}$ will remain positive definite as long as the solution exists, hence the desired obliqueness.
\end{proof}

\section{Estimate of $\udot$}
\begin{thm}\label{tderivthm}
We have the estimate
\begin{equation*}
\min_{t=0}\udot \leq \udot(x, t) \leq \max_{t=0}\udot,\  \forall x\in \Omega,\ 0\leq t<\tmax.
\end{equation*}
\end{thm}

\begin{proof}
Fix some $\ttil <\tmax$. By differentiating~\eqref{floweqn} in $t$ and writing $v(x, t) = \udot(x, t)$, we find that
\begin{equation*}
Lv= w^{ij}(v_{ij}-\Dp A_{ij}v_k)-\Dp\Btil v_k-\vdot=0
\end{equation*}
while differentiating the boundary condition~\eqref{bdrycond} in $t$, we see that
\begin{equation*}
v_\beta=0\text{ on }\ombdry.
\end{equation*}
Let $\hbar$ be a normalized defining function for $\Omega$, ie. $\nabla \hbar=\nu$ and $\hbar=0$ on $\ombdry$, and $\hbar< 0$ on $\Omega$. Define $\vtil(x, t) = v(x, t)-\epsilon \hbar(x)-C_1t$ for some fixed $\epsilon >0$ and a constant $C_1>0$ to be determined. By Lemma~\ref{lemma: stays strictly c-convex}, $w^{ij}$ remains positive definite as long as the solution exists, so for any fixed $\epsilon >0$, on $\Omega \times [0, \ttil]$ we have
\begin{equation*}
L\vtil=Lv-\epsilon[w^{ij}(h_{ij}-\Dp A_{ij}\hbar_k)-\Dp\Btil \hbar_k]+C_1>0
\end{equation*}
for the choice of $C_1=C\epsilon(1+M)$ with a constant $C$ depending only on bounds on $\lvert \nabla h\rvert$, $\lvert \hess{\hbar}\rvert$, $\lvert \nabla_p A_{ij}\rvert$, and $\lvert \nabla_p\Btil\rvert$, and with 
\begin{equation*}
M=\sup_{(x, t)\in\Omega\times [0, \ttil]}{\tr{(w^{ij}(x, t))}}.
\end{equation*}
Thus, the maximum principle implies that a maximum for $\vtil$ must occur on the parabolic boundary of $\Omega \times [0, \ttil]$\\
Now if $x\in\ombdry$, we have $(v-\epsilon \hbar-C_1t)_\beta=-\epsilon \inner{\nabla \hbar}{\beta}=-\epsilon \inner{\beta}{\nu}<0$, by~\eqref{obliqueness}. However, if $\vtil$ has a maximum at some $x\in\ombdry$, we have that $\nabla(v-\epsilon \hbar-C_1t)=\chi\nu$ for some $\chi \geq 0$. This implies that $(v-\epsilon \hbar-C_1t)_\beta=\chi\inner{\beta}{\nu}\geq 0$, a contradiction. Thus, the maximum for $\vtil$ must occur when $t=0$. Recalling that $\hbar< 0$ on $\Omega$ we have

\begin{align*}
v(x, t)&\leq (v(x, t)-\epsilon \hbar(x)-C_1t)+C_1t=\vtil(x, t)+C_1t\\
&\leq \sup_{x\in\Omega}\vtil(x, 0)+C_1t\leq\sup_{x\in\Omega}{v(x, 0)}+\epsilon\sup_{x\in\Omega}{\lvert \hbar\rvert}+C\epsilon(1+M)\tmax.
\end{align*}

By letting $\epsilon \to 0$ we obtain the bound
\begin{equation*}
v(x, t) \leq \sup_{x\in\Omega}v(x, 0)
\end{equation*}
for $t\leq \ttil$, and since $\ttil < \tmax$ was arbitrary, we obtain
\begin{equation*}
\sup_{(x, t)\in\Omega\times[0, \tmax)}{v(x, t)} \leq \sup_{x\in\Omega}{v(x, 0)}.
\end{equation*}
The lower bound is similar.
\end{proof}

\section{Uniform Obliqueness of the Boundary Condition}
By Theorem~\ref{thm: extended defining functions}, we can choose $\hstar$ in a way such that $\hstar<0$ on $\omstarbdryepstar=\left\{y\in\omstar \vert \dist(y, \omstarbdry)<\epstar\right\}$, $\hstar= 0$ and $\nabla \hstar = \nustar$ on $\omstarbdry$, while we have
\begin{equation}
[D_{ij}\hstar(y)-c^{k, l}c_{l, ij}(x, y)D_k\hstar(y)]\xi_i\xi_j \geq \delstar_0\lvert \xi \rvert^2
\label{globalcstarconvex}
\end{equation}
for any $x\in\Omega$, $y\in\omstarbdryepstar$, and $\xi \in \R^n$. Additionally, we take $G$ as defined in Corollary~\ref{cor: construction of G}\\

\begin{rmk}\label{rmk: we can replace hstar}
Note that the boundary condition~\eqref{bdrycond}
\begin{equation*}
\Gbar(x, \gradu) =0\text{ on }\ombdry
\end{equation*}
is equivalent to 
\begin{equation*}
G(x, \gradu)=0\text { on }\ombdry.
\label{altbdrycond}
\end{equation*}

Also, since $\nabla \hstar = \nabla \hstarbar = \nustar$ on $\omstarbdry$, we have that $G_p(x, \gradu)=\Gbar_p(x, \gradu)=\beta(x,t)$ on $\omstarbdry$.
\end{rmk}
\begin{thm}\label{unifobliquethm}
We have uniform strict obliqueness of the boundary condition~\eqref{bdrycond} for $t\in[0, \tmax)$, i.e.
\begin{equation*}
\inner{\Gbar_p(x, \gradu(x, t))}{\nu(x)} \geq C >0,\ x\in \ombdry,\ 0\leq t<\tmax
\label{unifobliqueness}
\end{equation*}
where $C$ depends on $\Omega$, $\omstar$, $B$, $c$, and $u_0$ but is independent of $t$ and $\tmax$. 
\end{thm}

\begin{proof}
We take $\ttil < \tmax$, and then find a $C$ as above that is independent of $\ttil$.
Continuing the calculations from Theorem~\ref{obliquethm}, we find
\begin{equation}
w_{kl}c^{i, k}\hstarbar_ic^{j, l}\hstarbar_j = \chi \nu_l(c^{j, l}\hstarbar_j)=\chi\nu_l\beta_l=\chi\inner{\beta}{\nu}.
\label{urbas2}
\end{equation}

Since $\chi \neq 0$, we can now combine~\eqref{urbas1} and~\eqref{urbas2}, and use Remark~\eqref{rmk: we can replace hstar} to write
\begin{align}
\inner{\beta}{\nu}^2 &= (w^{kl}\nu_k\nu_l)(w_{mn}c^{i, m}\hstarbar_ic^{j, n}\hstarbar_j)\notag\\
&=(w^{kl}\nu_k\nu_l)(w_{mn}c^{i, m}\hstar_ic^{j, n}\hstar_j)
\label{urbas3}
\end{align}
for $x\in\ombdry$.
We will proceed by bounding the two terms on the right from below.

Define the linearized operator by,
\begin{equation}\label{linearized}
L\phi(x, t) = -\phidot(x, t) + w^{ij}(x, t)(\phi_{ij}(x, t)-\Dp A_{ij}(x, T(x, t))\phi_k(x, t)).
\end{equation}
Let $(x_0, t_0)$ be a point where $\inner{\beta(x, t)}{\nu(x)}$ achieves its minimum on $\ombdry \times [0, \ttil)$. We also define $F(x, p) = \inner{G_p(x, p)}{\nu(x)} - \kappa G(x, p)$ and $v(x, t) = F(x, \gradu(x, t))$, where $\nu(x)$ is an extension of the outward normal to a neighborhood of $\ombdry$ using the function $h$. Since $G(x, \gradu(x, t))=\hstar(T(x, t))=0$ for $x\in \ombdry$, $v(x, t)$ restricted to $\ombdry \times [0, \ttil]$ achieves its minimum at the same point $(x_0, t_0)$. Let $\omep = \Omega \cap B_\epsilon(x_0)$ with $\epsilon$ chosen small enough so that $h < 0$ on $\omep$.
Then we calculate
\begin{align}
Lv &= -\Fp \udot_k + w^{ij}(\Fxx[i]{j}+2\Fxp{i}u_{kj}+\Fp u_{kij}+\Fpp{l}u_{ki}u_{lj}\notag\\
&\qquad-\Dp[m]A_{ij}(\Fx[m]-\Fp u_{km}))\notag\\
&=\Fp Lu_k + w^{ij}(\Fxx[i]{j}+2\Fxp{i}u_{kj}+\Fpp{l}u_{ki}u_{lj}-\Dp[m]A_{ij}\Fx[m])\notag\\
&=\Fp (w^{ij}\Dx A_{ij}+\Dx\Btil+\Dp[l]\Btil u_{lk})\notag\\
&\qquad+w^{ij}(\Fxx[i]{j}+2\Fxp{i}u_{kj}+\Fpp{l}u_{ki}u_{lj}-\Dp[m]A_{ij}\Fx[m])\notag\\
&=\Fp (w^{ij}\Dx A_{ij}+\Dx\Btil+\Dp[l]\Btil (w_{lk}+A_{lk}))\notag\\
&\qquad+w^{ij}(\Fxx[i]{j}+2\Fxp{i}(w_{kj}+A_{kl})+\Fpp{l}(w_{ki}w_{lj}+2w_{ki}A_{lj}+A_{ki}A_{lj})\notag\\
&\qquad-\Dp[m]A_{ij}\Fx[m])\notag\\
&\leq C(1+\tr{(w^{ij})})+2\Fxp{i}w^{ij}w_{kj}+w^{ij}\Fpp{l}w_{ki}w_{lj}+2\Fpp{l}w^{ij}w_{ki}A_{lj}\notag\\
&\qquad+\Fp(\Dp[l]\Btil) w_{lk}\notag\\
&\leq C(1+\tr{(w^{ij})})+\Fpp{l}w_{kl}+\Fp(\Dp[l]\Btil) w_{lk}\label{Lvcalc}
\end{align}
where $C$ is a constant depending on $\Omega$, $\omstar$, $\kappa$, $u_0$, the cost $c$, and $B$. The expression for $Lu_k$ comes from differentiation of the equation~\eqref{floweqn}.

By Corollary~\ref{cor: construction of G}, we have that $\Gpp{l}\xi_k\xi_l\geq \deltil \lvert\xi\rvert^2$, so we have 
\begin{align}
\Fpp{l}w_{kl} &\leq \Dpp{l}\inner{\beta}{\nu}w_{kl}-\kappa \Gpp{l}w_{kl}\notag\\
&=(C_0-\deltil\kappa)\tr{(w_{ij})}.
\label{Fppest}
\end{align} 
Thus by choosing $\kappa$ large enough, we will have by~\eqref{Lvcalc},
\begin{equation}
Lv \leq C(1+\tr{(w^{ij})})\leq C\tr{(w^{ij})}.
\label{Lvest}
\end{equation}
The second inequality comes from the arithmetic geometric mean inequality, combined with the fact that
\begin{equation*}
\det{w^{ij}}=(\det{w_{ij}})^{-1}=(B+e^{\udot})^{-1} \geq \frac{1}{C}
\end{equation*}
from bounds on $f$, $g$, $\hessxy{c}$ and Theorem~\ref{tderivthm}.\\
By Theorem~\ref{thm: extended defining functions} we find that
\begin{equation}
Lh = w^{ij}(D_{ij}h-c^{l, k}c_{ij, l}(x, T)D_kh)\geq \delta_0\tr{(w^{ij})}
\label{Lhest}
\end{equation}
We fix a point $y_0 \in \omstar$ and consider the function
\begin{align}\label{definition of theta}
\Theta(x,t) &= (v(x, t)-v(x_0, t_0))+\inner {\alpha}{\nabla_yc(x, y_0) - p_0}\notag\\
&\qquad+C_1\lvert x-x_0\rvert^2 -C_2h
\end{align}
where $p_0 = \nabla_yc(x_0, y_0)$, $C_1$ and $C_2$ are positive constants, and $\alpha \in \R^n$ are to be determined. The idea is that we will adjust $\alpha$ to obtain $\Theta \geq 0$ on $\omep \times \{0\}$, $C_1$ to obtain $\Theta \geq 0$ on $\partial(\omep) \times (0, \tmax]$, and $C_2$ to bound the term involving $\tr{(w^{ij})}$ in $Lv$ so that $L\Theta \leq 0$ everywhere. Calling the map $\Psi(x) = \nabla_yc(x, y_0)$, which is a diffeomorphism by conditions~\eqref{twist} and~\eqref{nondeg}, we have by the uniform $c$-convexity of $\Omega$ that $\Omtil = \Psi(\Omega)$ is a strictly convex set in $\R^n$. Hence, there is a supporting hyperplane to $\Omtil$ at $p_0=\Psi(x_0)$, and we may assume that the normal direction to the supporting hyperplane away from $\Omtil$ is given by $e_n$. Write $\vtil(p, t) = v(\Psi^{-1}(p), t)$, $P(p)$ for the projection of $p$ onto $\Omtilbdry$ in the $e_n$ direction, and $P_0(p)$ for the orthogonal projection of $p$ onto the supporting hyperplane to $\Omtil$. Then, since $\vtil(p, 0) - \vtil(p_0, t_0)\geq 0$ for $p\in\Omtilbdry$ we have that
\begin{align*}
\vtil(p, 0) - \vtil(p_0, t_0) &\geq \vtil(P(p), 0)-\vtil(p_0, t_0)-\sup\lvert\nabla\vtil(\cdot, 0)\rvert\lvert p-P(p)\rvert\\
&\geq -\sup\lvert\nabla\vtil(\cdot, 0)\rvert\lvert p-P_0(p)\rvert\\
&\geq -\inner{\alpha}{p-p_0}
\end{align*}
if $\alpha$ is a sufficiently large multiple of $-e_n$. Additionally, with this choice, $\inner{\alpha}{p-p_0}\geq 0$ on the half space defined by the supporting hyperplane to $\Omtil$ at $p_0$, hence also on $\Omtil$, and thus $\inner{\alpha}{\Psi(x)-p_0}\geq 0$ on $\omep$. By positivity of the terms $-h$ and $\lvert x-x_0\rvert^2$, we find that $\Theta \geq 0$ on $\omep \times \{0\}$ as desired, and $\alpha$ depends only on $\vtil$ at $t=0$. Next, on $\partial(\omep) \times (0, \tmax]$ we have that 
\begin{align*}
\Theta &\geq C_1\lvert x-x_0\rvert^2+(v(x, t)-v(x_0, t_0))\\
&\geq 0
\end{align*}
for $C_1$ large depending on $\epsilon$ and upper bounds on $\vert v\rvert$, which in turn depend on $\hstar$ and $c$.

Finally, since 
\begin{equation*}
L(\inner {\alpha}{\nabla_yc(x, y_0) - p_0}) =w^{ij}c_{ij, k}\alpha_k-w^{ij}(\Dp A_{ij})c_{i,k}\alpha_k\leq C\tr{(w^{ij})}
\end{equation*}
and
\begin{equation*}
L(\lvert x-x_0\rvert^2)=\tr{(w^{ij})}-w^{ij}(\Dp A_{ij})(x-x_0)_k\leq C\tr{(w^{ij})}
\end{equation*}
we calculate from~\eqref{Lvest} and~\eqref{Lhest} that 
\begin{equation*}
L\Theta \leq C\tr{(w^{ij})} -\delta_0C_2 \tr{(w^{ij})} \leq 0
\end{equation*}
for $C_2$ sufficiently large.
Thus noting that $\Theta(x_0, t_0) = 0$, by the comparison principle we find that $\inner{\nabla \Theta}{\nu} \leq 0$ or
\begin{align}
\inner{\nabla v}{\nu}&\leq C\lvert\inner{\alpha}{\nu}\rvert+C_2\lvert\inner{\nabla h}{\nu}\rvert\leq C.
\label{Dvest}
\end{align}
Now we calculate at $(x_0, t_0)$, using the formulae from Remark~\eqref{rmk: formulae for inverse of c_x,y}, 
\begin{align*}
D_i\inner{\beta}{\nu} &= D_i(\hstar_lc^{l, k}\nu_k)\notag\\
&= \hstar_{lm}(D_iT^m)c^{l, k}\nu_k+\hstar_l(c^{l, k}_i+c^{l, k}_{\ ,m}D_iT^m)\nu_k+\hstar_lc^{l, k}D_i\nu_k\notag\\
&= \hstar_l(c^{l, k}D_i\nu_k-c^{l,s}c^{r,k}c_{is, r}\nu_k)\notag\\
&\qquad+(D_iT^m)(\hstar_{lm}c^{l, k}\nu_k-\hstar_lc^{l,s}c^{r,k}c_{s, mr}\nu_k)\notag\\
&= \hstar_lc^{l, j}(D_i\nu_j-c^{r,k}c_{ij, r}\nu_k)\notag\\
&\qquad+(D_iT^m)c^{l, k}\nu_k(\hstar_{lm}-\hstar_rc^{r,s}c_{s, ml})
\end{align*}
Multiply by $\hstar_qc^{q, i}$ and sum to obtain
\begin{align*}
\hstar_qc^{q, i}D_i\inner{\beta}{\nu}&= \hstar_qc^{q, i}\hstar_lc^{l, j}(D_i\nu_j-c^{r,k}c_{ij, r}\nu_k)\notag\\
&\qquad+\hstar_qc^{q, i}(D_iT^m)c^{l, k}\nu_k(\hstar_{lm}-\hstar_rc^{r,s}c_{s, ml})\notag\\
&= \hstar_qc^{q, i}\hstar_lc^{l, j}(D_i\nu_j-c^{r,k}c_{ij, r}\nu_k)\notag\\
&\qquad+\hstar_qc^{q, i}c^{m, p}w_{pi}c^{l, k}\nu_k(\hstar_{lm}-\hstar_rc^{r,s}c_{s, ml})\notag\\
&= \hstar_qc^{q, i}\hstar_lc^{l, j}(D_i\nu_j-c^{r,k}c_{ij, r}\nu_k)\notag\\
&\qquad+c^{m, p}\chi\nu_pc^{l, k}\nu_k(\hstar_{lm}-\hstar_rc^{r,s}c_{s, ml})\text{, by}~\eqref{urbas0}\notag\\
&\geq \delta_0 \sum_i \lvert \hstar_qc^{q, i}\rvert^2+\chi\delstar_0 \sum_i \lvert \nu_qc^{q, i}\rvert^2\text{, by Theorem}~\ref{thm: extended defining functions}\notag\\
&\geq \delta_0 \sum_i \lvert \hstar_qc^{q, i}\rvert^2\geq C_0\text{, by}~\eqref{nondeg}.
\end{align*}
Since $v(\cdot, t_0)$ takes its minimum on $\ombdry$ at $x=x_0$, $\nabla v(x_0, t_0) = a_0\nu(x_0)$ for some $a_0\leq C$ from~\eqref{Dvest}. Thus at $(x_0, t_0)$,
\begin{align*}
a_0\inner{\beta}{\nu}+ \kappa\hstar_qc^{q, i}\hstar_sc^{s, k}w_{ki}&=\hstar_qc^{q, i}(a_0\nu_i+\kappa\hstar_sc^{s, k}w_{ki})\\
&=\hstar_qc^{q, i}(a_0\nu_i+\kappa D_i(\hstar(T)))\\
&=\hstar_qc^{q, i}D_i\inner{\beta}{\nu}\\
&\geq C_0.\\
\end{align*}
Now if $a_0\leq 0$, we can throw that term away and obtain the desired bound. If $a_0>0$ we have two cases. If $\inner{\beta}{\nu}\geq \frac{C_0}{2a_0}\geq\frac{C_0}{2C}$, again we have the desired bound already. Otherwise, we see that
\begin{align*}
\kappa\hstar_qc^{q, i}\hstar_sc^{s, k}w_{ki}&\geq-a_0\inner{\beta}{\nu}+C_0\\
&\geq \frac{C_0}{2}.
\label{obliquelower1}
\end{align*}

In this last case, we continue by estimating $w^{ij}\nu_i\nu_j$ from below. To do this we use the transportation problem in the opposite direction. 

Taking $u^*$ as in Lemma~\ref{reverseproblem}, and again taking $h$ and $\hstar$ constructed in Theorem~\ref{thm: extended defining functions}, we find that for $x\in\ombdry$, 
\begin{equation*}
\inner{\beta}{\nu}(x, t)=c^{k,l}(x, T(x, t))h_l(x)\hstar_k(T(x, t))=\inner{\beta^*}{\nustar}(T(x, t)), 
\end{equation*}
if we define
\begin{equation*}
\beta^*(y)=G^*_q(\nabla u^*, y)
\end{equation*}
where $G^*(q, y) = h(X(q, y))$. This implies that if we take $(x_0, t_0)$ again as the point where the minimum to $\inner{\beta}{\nu}(x, t)$ occurs, $\inner{\beta^*}{\nu}$ has its minimum over $\omstarbdry \times [0, \ttil]$ at $(y_0, t_0) = (T^*(x_0, t_0), t_0)$. If we write $w^*_{kl}(y, t) = u^*_{kl}(y, t) - c_{\ ,kl}(T^*(y, t), t)$, we find that for $(x, y, t)=(x, T(x, t), t)$, we have $c^{k, i}w^*_{jk}=D_jT^{*i}=(D_jT^i)^{-1}=(c^{i, l}w_{lj})^{-1}=w^{il}c_{l, j}$. Hence
\begin{align}
w^*_{kl}c^{k, i}c^{l,j}&=w^{ij}\label{wupper}\\
w_{kl}c^{k, i}c^{l,j}&=w^{*ij}\label{wlower}
\end{align}
and
\begin{equation*}
w^*_{kl}c^{k,m}c^{l,n}(T^*(y,t), y)\nu_m\nu_n(T^*(y,t))=w^{ij}(T^*(y, t), y)\nu_i\nu_j(T^*(y, t)).
\end{equation*}
Thus with a similar proof, we estimate $w^*_{kl}c^{k,m}c^{l,n}\nu_m\nu_n$ from below, and hence $w^{ij}\nu_i\nu_j$ and by~\eqref{urbas3} the lower bound of $\inner{\beta}{\nu}$ is established.
\end{proof}

\section{Interior $C^2$ estimates}
\begin{thm}\label{intc2est}
For each $t\in[0, \tmax)$ we have the bound
\begin{equation}
\sup_{x\in\Omega}{\lvert \hess{u(x, t)}\rvert} \leq C(1+M)
\end{equation}
where $\displaystyle M=\sup_{x\in\ombdry}\lvert \hess{u(x, t)}\rvert$, for a constant $C$ depending on $\Omega$, $\omstar$, $B$, $c$, and $u_0$, but is independent of $t$.
\end{thm}

\begin{proof}
Assume again, that $\ttil <\tmax$. Redefine the linearized operator as
\begin{align*}
L\phi(x, t)& = -\phidot(x, t) + w^{ij}(x, t)(\phi_{ij}(x, t)-\Dp A_{ij}(x, T(x, t))\phi_k(x, t))\notag\\
&\qquad-(\Dp\Btil)\phi_k(x, t)
\end{align*}
and let
\begin{equation*}
v(x, t, \xi)=\log{(w_{\xi\xi}(x, t))}+a\lvert\gradu(x, t)\rvert^2
\end{equation*}
for some fixed constant $a$, where $\xi$ is a unit vector. We differentiate~\eqref{floweqn} twice in $x$, in the $\xi$ direction to obtain:
\begin{align*}
w^{ij}[D_{ij}u_\xi-D_\xi A_{ij}-(\Dp A_{ij})D_ku_\xi]-\udot_{\xi}=D_\xi\Btil + (\Dp\Btil) D_k u_\xi
\end{align*}
and since $D_\xi w^{ij} = -w^{il}(D_\xi w_{lk})w^{kj}$,
\begin{align}\label{wxixi}
&w^{ij}[D_{ij}u_{\xi\xi}-D_{\xi\xi}A_{ij}-2((D_\xi\Dp A_{ij})D_ku_\xi)-(\Dpp{l}A_{ij})D_ku_\xi D_lu_\xi\notag\\
&\qquad-(\Dp A_{ij})D_ku_{\xi\xi}] -\udot_{\xi\xi}-w^{il}w^{jk}D_\xi w_{ij}D_\xi w_{lk}\notag\\
&\qquad=D_{\xi\xi}\Btil + 2(D_\xi\Dp\Btil) D_k u_\xi+(\Dp\Btil) D_ku_{\xi\xi}+(\Dpp{l}\Btil) D_ku_\xi D_lu_\xi.
\end{align}
Now fix $(x, t)$ and let $\{\vvec^k\}$ be a set of orthonormal eigenvectors for $w_{ij}(x, t)$ with eigenvalues $\lambda_k$. We write $(\vvec^k)_i$ to denote the $i$th component of $\vvec_k$. Using~\eqref{MTW} we obtain at $(x, t)$, 
\begin{align*}
&w^{ij}(\Dpp{l}A_{ij})w_{k\xi} w_{l\xi}\notag\\
&=\sum_p(\Dpp{l}A_{ij})(\vvec^p)_i(\vvec^p)_j(\lambda_p)^{-1}[\lambda_p\inner{\xi}{\vvec^p}(\vvec^p)_k\notag\\
&\qquad+\sum_{q\neq p}\lambda_q\inner{\xi}{\vvec^q}(\vvec^q)_k]\cdot[\lambda_p\inner{\xi}{\vvec^p}(\vvec^p)_l+\sum_{r\neq p}\lambda_r\inner{\xi}{\vvec^r}(\vvec^r)_l]\notag\\
&\geq \sum_p\sum_{q\ or\ r=p}(\Dpp{l}A_{ij})(\vvec^p)_i(\vvec^p)_j(\lambda_p)^{-1}[\lambda_q\inner{\xi}{\vvec^q}(\vvec^q)_k][\lambda_r\inner{\xi}{\vvec^r}(\vvec^r)_l]\notag\\
&= \sum_{p, q}(\Dpp{l}A_{ij})(\vvec^p)_i(\vvec^p)_j\lambda_q\xi_k\xi_l\notag\\
&\geq -C\sum_{q}\lambda_q\notag\\
&= -C\tr{(w_{ij})}.
\end{align*}
Using this and~\eqref{wxixi}we calculate
\begin{align*}
Lu_{\xi\xi}&=w^{ij}(D_{ij}u_{\xi\xi}-\Dp A_{ij})D_ku_{\xi\xi})-\udot_{\xi\xi}-(\Dp\Btil) D_ku_{\xi\xi}\notag\\
&=w^{il}w^{jk}D_\xi w_{ij}D_\xi w_{lk}+w^{ij}(D_{\xi\xi}A_{ij}+2((D_\xi\Dp A_{ij})D_ku_\xi)\notag\\
&\qquad+(\Dpp{l}A_{ij})D_ku_\xi D_lu_\xi)\notag\\
&\qquad+D_{\xi\xi}\Btil + 2(D_\xi\Dp\Btil) D_k u_\xi+(\Dpp{l}\Btil) D_ku_\xi D_lu_\xi\notag\\
&=w^{il}w^{jk}D_\xi w_{ij}D_\xi w_{lk}+w^{ij}[D_{\xi\xi}A_{ij}+2((D_\xi\Dp A_{ij})(w_{k\xi}+A_{k\xi}))\notag\\
&\qquad+(\Dpp{l}A_{ij})(w_{k\xi} w_{l\xi}+2w_{k\xi}A_{l\xi}+A_{k\xi}A_{l\xi})]+D_{\xi\xi}\Btil \notag\\
&\qquad+ 2(D_\xi\Dp\Btil)(w_{k\xi}+A_{k\xi})\notag\\
&\qquad+(\Dpp{l}\Btil)(w_{k\xi} w_{l\xi}+2w_{k\xi}A_{l\xi}+A_{k\xi}A_{l\xi}))\notag\\
&\geq w^{il}w^{jk}D_\xi w_{ij}D_\xi w_{lk}+w^{ij}(\Dpp{l}A_{ij})w_{k\xi} w_{l\xi}\notag\\
&\qquad-C(1+\tr{(w^{ij})}+\tr{(w^{ij})}\tr{(w_{ij})}+\tr{(w_{ij})}+\tr{(w_{ij})}^2)\notag\\
&\geq w^{il}w^{jk}D_\xi w_{ij}D_\xi w_{lk}\notag\\
&\qquad-C(1+\tr{(w^{ij})}+\tr{(w^{ij})}\tr{(w_{ij})}+\tr{(w_{ij})}+\tr{(w_{ij})}^2).
\end{align*}
Also,
\begin{align*}
LA_{\xi\xi}&=w^{ij}[\Dxx[i]{j} A_{\xi\xi}+2(\Dxp{i} A_{\xi\xi})u_{kj}+(\Dpp{l}A_{\xi\xi})u_{ki}u_{lj}-\Dp A_{ij}(\Dx A_{\xi\xi})]\notag\\
&\qquad+\Dp A_{\xi\xi}[w^{ij}(u_{kij}-(\Dp[l]A_{ij})u_{kl})-\udot_k]-(\Dp\Btil)D_kA_{\xi\xi}\notag\\
&=w^{ij}[\Dxx[i]{j} A_{\xi\xi}+2(\Dxp{i} A_{\xi\xi})(w_{kj}+A_{kl})\notag\\
&\qquad+(\Dpp{l}A_{\xi\xi})(w_{ki}w_{lj}+2w_{ki}A_{lj}+A_{ki}A_{lj})\notag\\
&\qquad-\Dp A_{ij}(\Dx A_{\xi\xi})]+\Dp A_{\xi\xi}[w^{ij}\Dx A_{ij}+\Dx\Btil\notag\\
&\qquad+(\Dp[l]\Btil)(w_{lk}+A_{lk})]-(\Dp\Btil)D_kA_{\xi\xi}\notag\\
&\geq-C(1+\tr{(w^{ij})}+\tr{(w^{ij})}\tr{(w_{ij})}+\tr{(w_{ij})})
\end{align*}
where the expression for $w^{ij}(u_{kij}-(\Dp[l]A_{ij})u_{kl})-\udot_k$ is from differentiating~\eqref{floweqn}. Thus we have
\begin{align}
Lw_{\xi\xi} &\geq w^{il}w^{jk}D_\xi w_{ij}D_\xi w_{lk}\notag\\
&\qquad-C(1+\tr{(w^{ij})}+\tr{(w^{ij})}\tr{(w_{ij})}+\tr{(w_{ij})}+\tr{(w_{ij})}^2).
\label{Lwxixiest}
\end{align}
Differentiating $v$, we have
\begin{equation*}
D_iv=\frac{D_iw_{\xi\xi}}{w_{\xi\xi}}+2au_lu_{li}
\end{equation*}
\begin{equation*}
D_{ij}v=\frac{D_{ij}w_{\xi\xi}}{w_{\xi\xi}}-\frac{D_iw_{\xi\xi}D_jw_{\xi\xi}}{(w_{\xi\xi})^2}+2au_{lj}u_{li}+2au_lu_{lij}
\end{equation*}
and
\begin{equation*}
\vdot = \frac{\wdot_{\xi\xi}}{w_{\xi\xi}}+2au_k\udot_k.
\end{equation*}

Suppose that $v$ takes its maximum at some point $(x_0, t_0)$ in $(\Omega \times (0, \ttil))\cup (\Omega \times \{\ttil\})$ and some $\xi$. There we have $\hessx v \leq 0$, $\nabla v = 0$, and $\vdot \geq 0$, hence by differentiation of the equation~\eqref{floweqn}, writing $u_{ij}=w_{ij}+A_{ij}$ again, assuming $w_{\xi\xi}\geq 1$ and using~\eqref{Lwxixiest} we have
\begin{align*}
0&\geq Lv =\frac{Lw_{\xi\xi}}{w_{\xi\xi}}-\frac{w^{ij}D_iw_{\xi\xi}D_jw_{\xi\xi}}{(w_{\xi\xi})^2}+w^{ij}2au_{lj}u_{li}\notag\\
&\qquad+2au_l[w^{ij}(u_{lij}-(\Dp A_{ij})u_{kl})-\udot_l-(\Dp \Btil)u_{lk}]\notag\\
&\geq \frac{w^{il}w^{jk}D_\xi w_{ij}D_\xi w_{lk}}{w_{\xi\xi}}-\frac{w^{ij}D_iw_{\xi\xi}D_jw_{\xi\xi}}{(w_{\xi\xi})^2}+w^{ij}2aw_{lj}w_{li}+2au_l\Dx[l]\Btil\notag\\
&\qquad-C(w_{\xi\xi})^{-1}(1+\tr{(w^{ij})}+\tr{(w^{ij})}\tr{(w_{ij})}+\tr{(w_{ij})}+\tr{(w_{ij})}^2)\notag\\
&\geq \frac{w^{il}w^{jk}D_\xi w_{ij}D_\xi w_{lk}}{w_{\xi\xi}}-\frac{w^{ij}D_iw_{\xi\xi}D_jw_{\xi\xi}}{(w_{\xi\xi})^2}+(2a-C)\tr{(w_{ij})}\notag\\
&\qquad-C_a-C\tr{(w^{ij})}.
\end{align*}
For the first two terms above, we change coordinates so $w_{ij}$ is diagonal at $(x_0, t_0)$ with $\xi=e_1$ and calculate
\begin{align*}
&\frac{w^{il}w^{jk}D_\xi w_{ij}D_\xi w_{lk}}{w_{\xi\xi}}-\frac{w^{ij}D_i w_{\xi\xi}D_j w_{\xi\xi}}{(w_{\xi\xi})^2}\notag\\
&=\frac{w^{ii}w^{jj}(D_1 w_{ij})^2}{w_{11}}-\frac{w^{ii}(D_iw_{11})^2}{(w_{11})^2}\notag\\
&\geq \frac{\sum_{i>1}[2w^{ii}(D_1w_{1i})^2-w^{ii}(D_iw_{11})^2]}{(w_{11})^2}\notag\\
&=\frac{\sum_{i>1}w^{ii}[(D_iw_{11})^2+2(D_1w_{1i}-D_iw_{11})(D_1w_{1i}+D_iw_{11})]}{(w_{11})^2}\notag\\
&=\frac{\sum_{i>1}w^{ii}[(D_iw_{11})^2+2(u_{11i}-D_1A_{1i}-u_{i11}+D_iA_{11})(D_iw_{11}+u_{11i}-D_1A_{1i}-D_iA_{11}+D_iA_{11})]}{(w_{11})^2}\notag\\
&=\frac{\sum_{i>1}w^{ii}[(D_iw_{11})^2+2(D_iA_{11}-D_1A_{1i})(2D_iw_{11}+D_iA_{11}-D_1A_{1i})]}{(w_{11})^2}\notag\\
&=\frac{\sum_{i>1}w^{ii}[(D_iw_{11})^2+4(D_iA_{11}-D_1A_{1i})(D_iw_{11})+2(D_iA_{11}-D_1A_{1i})^2]}{(w_{11})^2}\notag\\
&\geq\frac{\sum_{i>1}w^{ii}[-2(D_iA_{11}-D_1A_{1i})^2]}{(w_{11})^2}\notag\\
&\geq\frac{-\tr{(w^{ij})}C(1+\tr{(w_{ij})})^2}{(w_{11})^2}\notag\\
&\geq-C\tr{(w^{ij})}.
\end{align*}
Hence we have
\begin{equation*}
C_a+C\tr{(w^{ij})})\geq(2a-C)\tr{(w_{ij})}
\end{equation*}
so choosing $a$ large enough, we obtain at $(x_0, t_0)$, for any $\epsilon >0$
\begin{equation*}
C_\epsilon+\epsilon\sup_{\Omega\times (0, \ttil]}{\tr{(w^{ij})}} \geq \tr{(w_{ij})}.
\end{equation*}
Now again, using the solution $u^*$ defined in Lemma~\ref{reverseproblem}, using~\eqref{wupper}, and applying the same calculations, we find that
\begin{align*}
\tr{(w^{ij})}&= \tr{(w^*_{kl}c^{k,i}c^{l,j})}\notag\\
&\leq C\tr{(w^*_{ij})}\notag\\
&\leq C_\epsilon+\epsilon\tr{(w^{*ij})}\notag\\
&\leq C_\epsilon+\epsilon\tr{(w_{kl}c^{k, i}c^{l, j})}\notag\\
&\leq C_\epsilon+C\epsilon\sup_{\Omega\times (0, \ttil]}\tr{(w_{ij})}.
\end{align*}
Thus combining the above, for $\epsilon$ small enough we obtain 
\begin{equation*}
\sup_{\Omega\times (0, \ttil]}\tr{(w_{ij})}\leq C.
\end{equation*}
If the max for $v$ occurs at $x_0\in\ombdry$ or $t_0=0$, we can simply add the terms $\displaystyle \sup_{x\in\Omega}\lvert\hess{u(x, 0)}\rvert +M$ to the right hand side of the estimate, and we absorb the first term into $C$ by allowing its dependence on $u_0$.
\end{proof}

\section{Boundary $C^2$ estimates}
By tangentially differentiating the boundary condition $\hstar(T(x, t))=0$ on $\ombdry$, we see that $u_{\beta\tau}=0$ for any $\tau$ tangential to $\ombdry$.
\begin{thm}\label{ubetabeta}
For each $t\in[0, \tmax)$,
\begin{equation*}
u_{\beta\beta}(x, t) \leq C(1+\Mtil)^{\frac{n-2}{n-1}}
\end{equation*}
for some $C$ that depends on $\Omega$, $\omstar$, $B$, $c$, and $u_0$, but is independent of $t$, where $\displaystyle \Mtil=\sup_{x\in\Omega}\lvert\hess{u(x, t)}\rvert$.
\end{thm}

\begin{proof}
We take the linearized operator $L$ by~\eqref{linearized}, as in the proof to Theorem~\ref{unifobliquethm}, and $v(x, t)=G(x, \gradu(x, t))$ with $G(x, p)$ constructed in Corollary~\ref{cor: construction of G}. By the same calculation as~\eqref{Lvcalc}, with $G(x, p)$ in place of $F(x, p)$ we see that,
\begin{align*}
Lv &\leq C(1+\tr{(w^{ij})})+\Gpp{l}w_{kl}+\Gp\Dp[l]\Btil w_{lk}\\
&\leq C(1+\tr{(w^{ij})}+\tr{(w_{ij})}).\\
\end{align*}
Now since $w_{ij}$ is always positive definite, and $\det{(w_{ij})}=Be^{\udot}$ has an upper bound by Theorem~\ref{tderivthm}, we can see that
\begin{align*}
(\tr{(w_{ij})})^\frac{1}{n-1}&=(\sum{\lambda_k})^\frac{1}{n-1}=(\det{(w_{ij})})^\frac{1}{n-1}(\sum\frac{1}{\lambda_1\ldots\lambda_{k-1}\lambda_{k+1}\ldots\lambda_n})^\frac{1}{n-1}\\
&\leq C\lambda_1^{-1}\\
&\leq C\tr{(w^{ij})}
\end{align*}
where $0<\lambda_1\leq \dots \leq\lambda_n$ are the eigenvalues of $w_{ij}$. Since $\frac{1}{\tr{w^{ij}}}\leq C$ as in the proof of Theorem~\ref{unifobliquethm},  we have 
\begin{equation*}
Lv\leq C\tr{(w^{ij})}\left(1+\frac{\tr{(w_{ij})}}{\tr{(w^{ij})}}\right)\leq C\tr{(w^{ij})}(1+\Mtil)^{\frac{n-2}{n-1}}
\end{equation*}
As in the proof of Theorem~\ref{unifobliquethm}, define $\Theta$ by~\eqref{definition of theta}, only with $C_1=C(1+\Mtil)^{\frac{n-2}{n-1}}$, and note that $v(\cdot, t)=0$ for all $x\in\ombdry$ so every such point is a minimum of $v$. Then using the same comparison argument, only in the direction of $\beta$ (which is permissible by the obliqueness condition), we obtain
\begin{equation}
u_{\beta\beta}-A_{\beta\beta}=w_{\beta\beta}=v_\beta\leq C(1+\Mtil)^{\frac{n-2}{n-1}}
\label{wbetabeta}
\end{equation}
giving the desired estimate.
\end{proof}

\begin{thm}\label{utangent}
For any $x\in\ombdry$, and for each $t\in[0, \tmax)$,
\begin{equation*}
\sup_{x\in\ombdry}\lvert \hess{u(x, t)}\rvert \leq C
\end{equation*}
for some $C$ that depends on $\Omega$, $\omstar$, $B$, $c$, and $u_0$, but is independent of $t$.
\end{thm}

\begin{proof}
Assume that $\ttil <\tmax$ and $\displaystyle \sup_{\ombdry \times [0,\ttil]}{\sup_{\lvert\xi\rvert=1,\inner{\xi}{\nu}=0}{w_{\xi\xi}(x, t)}}$ occurs at $\xi=e_1$ and some $(x_0, t_0)$, and $w_{11}(x_0, t_0)\geq 1$. We write $e_1 = \tau +b\beta$ where $b=\frac{\inner{\nu}{e_1}}{\inner{\beta}{\nu}}$ and $\tau = e_1-b\beta$. Note that $\inner{\tau}{\nu}=0$. Then we obtain, at any $x\in\ombdry$ and any $t$, using Theorem~\ref{wbetabeta}, Theorem~\ref{unifobliquethm}, and the fact that $u_{\beta\tau}=0$,
\begin{align*}
w_{11}&=w_{\tau\tau}+2bw_{\tau\beta}+b^2w_{\beta\beta}\notag\\
&\leq \lvert\tau\rvert^2w_{11}(x_0, t_0) +2bA_{\tau\beta}+b^2C(1+\Mtil)^{\frac{n-2}{n-1}}\notag\\
&\leq (1-2b\inner{\beta}{e_1}+b^2\lvert\beta\rvert^2)w_{11}(x_0, t_0) +2bA_{\tau\beta}+b^2C(1+\Mtil)^{\frac{n-2}{n-1}}\notag\\
&\leq (1-2b\inner{\beta}{e_1}+C\inner{\nu}{e_1}^2)w_{11}(x_0, t_0) +2bA_{\tau\beta}+C\inner{\nu}{e_1}^2(1+\Mtil)^{\frac{n-2}{n-1}}.
\end{align*}
With $G$ constructed in Corollary~\ref{cor: construction of G} and $h$ constructed in Theorem~\ref{thm: extended defining functions}, we can extend $\beta$ and $\nu$ to all of $\Omega$ using the formulas $\nu(x)=\nabla h(x)$ and $\beta(x, t) = \Gp(x, \gradu(x, t))$, which also extends $b$ and $\tau$ to all of $\Omega$. Since $\inner{\nu(x_0)}{e_1}^2=0$, we have $\nabla \inner{\nu(x)}{e_1}^2\vert_{x=x_0}=0$ and thus by Taylor expanding the last term on the right hand side above about $x_0$, we have that
\begin{equation}
\frac{w_{11}}{w_{11}(x_0, t_0)}-1+2b\inner{\beta}{e_1}-\frac{2bA_{\tau\beta}}{w_{11}(x_0, t_0)}\leq C(1+\Mtil)^{\frac{n-2}{n-1}}\lvert x-x_0\rvert^2
\label{bdrybound}
\end{equation}
for all $x$ near $x_0$.

We now follow a barrier construction as in the proof to Theorem~\ref{unifobliquethm}. Again, let $\omep = \Omega \cap B_\epsilon(x_0)$ with $\epsilon$ chosen small enough so that $h \leq 0$ on $\omep$. This time we consider the function $v(x, t) = \frac{w_{11}(x, t)}{w_{11}(x_0, t_0)}-1+2b\inner{\beta}{e_1}-\frac{2bA_{\tau\beta}}{w_{11}(x_0, t_0)}+\kappa G(x, \gradu(x, t))$ and
\begin{equation*}
\Theta(x,t) = C_1h-C_2\lvert x-x_0\rvert^2 +v(x, t) -\inner {\alpha}{\nabla_yc(x, y_0) - p_0}
\end{equation*}
where $p_0$ and $\alpha$ are determined the same way as in Theorem~\ref{unifobliquethm}. \\
First, note that from the positivity of $w^{ij}$, and writing $\wtil$ as the matrix square root of $w^{ij}$, we have
\begin{equation*}
w^{il}w^{jk}D_1 w_{ij}D_1 w_{lk}=w^{il}[(D_1 w_{ij})\wtil_{jm}][(D_1 w_{lk})\wtil_{km}]\geq 0
\end{equation*}
Additionally, if we let $M_w$, $\Mtil_w$ be $M$ and $\Mtil$ with $w$ in place of $u$, we find from Theorem~\ref{intc2est} that for $x\in\Omega$ and $t\leq\ttil$,
\begin{align*}
\tr{(w_{ij})}&\leq C\Mtil_w\leq C(1+\sup_{x\in\Omega}\lvert\hess{u(x, t)}\rvert)\\
&\leq C(1+\sup_{x\in\ombdry}\lvert\hess{u(x, t)}\rvert)=C(1+\sup_{x\in\ombdry}\lvert w_{ij}+A_{ij}\rvert)\\
&\leq C(1+w_{11}(x_0, t_0)) \leq Cw_{11}(x_0, t_0)
\end{align*}
Combining this with~\eqref{Lwxixiest}, we have that
\begin{equation*}
L\frac{w_{11}}{w_{11}(x_0, t_0)}\geq -C(1+\tr{(w^{ij})}+\tr{(w_{ij})})
\end{equation*}
Noting that $2b\inner{\beta}{e_1}-\frac{2bA_{\tau\beta}}{w_{11}(x_0, t_0)}$ is just a function of the form $F(x, \gradu)$, ~\eqref{Lvcalc} applies to give us
\begin{equation*}
L\left(2b\inner{\beta}{e_1}-\frac{2bA_{\tau\beta}}{w_{11}(x_0, t_0)}\right)\geq-C(1+\tr{(w^{ij})}+\tr{(w_{ij})}).
\end{equation*}
Thus by fixing a $\kappa$ large enough, and also using the calculation from~\eqref{Fppest}, we find that
\begin{equation*}
Lv\geq -C(1+\tr{(w^{ij})}).
\end{equation*}
Now since $h$ and $G$ are nonpositive, and by the choice of the linear term, along with~\eqref{bdrybound}, we find that on $\partial(\omep)$
\begin{equation*}
\Theta\leq C(1+\Mtil)^{\frac{n-2}{n-1}}\lvert x-x_0\rvert^2.
\end{equation*}
Finally, choosing $C_1$ large enough, we will find that $L\Theta \geq 0$ while $\Theta \leq 0$ on $\partial(\omep) \times (0, \tmax]$ and $\omep \times \{0\}$, thus by applying the comparison principle and differentiating in the direction of $-\beta$ we obtain $\inner{\nabla\Theta}{-\beta}\leq 0$ or
\begin{align}
-w_{11\beta}&\leq C(1+\Mtil)^{\frac{n-2}{n-1}}w_{11}(x_0, t_0)\notag\\
&\leq C(1+M)^{\frac{n-2}{n-1}}(C+M)\notag\\
&\leq C(1+M)^{\frac{2n-3}{n-1}}\leq C(1+M_w)^{\frac{2n-3}{n-1}}.
\label{w11betaest}
\end{align}
Now by differentiating the condition $G(x, \gradu(x, t))=0$ twice and using the strict positivity of $\Gpp{l}$ along with~\eqref{w11betaest}, at $(x_0, t_0)$ we get
\begin{align*}
\deltil\sum_k\lvert u_{k1}\rvert^2&\leq \Gpp{l} u_{k1}u_{l1}\notag\\
&=-\Gp u_{k11}-G_{x_1x_1}-2G_{x_1p_k}u_{k1}\notag\\
&\leq C(1+M_w)^{\frac{2n-3}{n-1}}
\end{align*}
Since
\begin{align*}
M_w^2&=(w_{11}(x_0, t_0))^2\leq\sum_k(w_{k1}(x_0, t_0))^2\\
&\leq \sum_k(C+\vert u_{k1}(x_0, t_0)\rvert)^2\\
&\leq C(1+M_w)+\sum_k(u_{k1}(x_0, t_0))^2\
\end{align*}
we see that
\begin{equation*}
M_w^2\leq C(1+M_w)^{\frac{2n-3}{n-1}}
\end{equation*}
or
\begin{equation*}
M_w^{\frac{2n-2}{2n-3}}-CM_w-C\leq 0
\end{equation*}
and thus 
\begin{equation*}
M\leq C+M_w \leq C
\end{equation*}
as desired.
\end{proof}

\section{Long time convergence of solutions to the flow equation}
By the uniform $C^2$ estimates on $u$, we find that our equation is uniformly parabolic, and the theory of \cite[Chapter 14]{Lieberman:1996:SOP} gives us $C^{2+\alpha}$ estimates on $u$, and hence a standard argument using the Arzel\`{a}-Ascoli theorem gives existence of a smooth solution $u$ for all times $t>0$.

Now fix some positive $t_0$ and write $v(x, t) = u(x, t)-u(x, t+t_0)$, and write $F(x, p,r)=\log{\det{(r_{ij}-A_{ij}(x, p))}}-\Btil(x, p)$. Since 
\begin{align*}
\vdot&=F(x, \gradu(x, t),\hess{u(x, t)})-F(x, \gradu(x, t), \hess{u(x, t+t_0)})\\
&\qquad+F(x, \gradu(x, t), \hess{u(x, t+t_0)})-F(x, \gradu(x, t+t_0), \hess{u(x, t+t_0)})
\end{align*}
we can use the mean value theorem to see that $\vdot=a^{ij}v_{ij}+b^iv_i$ for some functions $a^{ij}$ and $b^i$. We calculate that 
\begin{align}
a^{ij}&=\int_0^1\left[\nabla_{r_{ij}}F(x, s\gradu(x, t)+(1-s)\gradu(x, t+t_0), s\hess{u(x, t)}\right.\notag\\
&\qquad\left.+(1-s)\hess{u(x, t+t_0)})\right] ds\notag\\
&=\int_0^1[s\hess{u(x, t)}+(1-s)\hess{u(x, t+t_0)}-A(x, s\gradu(x, t)\notag\\
&\qquad+(1-s)\gradu(x, t+t_0))]^{ij} ds.
\label{aij}
\end{align}
Now, the equation that $w_{ij}$ satisfies combined with bounds on $\udot$ and $B$ give us a lower bound on $\tr{w_{ij}}$. Combined with the uniform upper bound on $w_{ij}$, we obtain a strictly positive lower bound on the smallest eigenvalue of $w_{ij}$, uniform in $t$ and $x$. Since $u\in C^{2+\alpha}(\Omega \times \R)$, we have that 
\begin{equation*}
\lVert u(\cdot, t)-u(\cdot, t+t_0)\rVert_{C^2(\Omega)}\leq C t_0^\alpha,
\end{equation*}
and thus by taking $t_0$ sufficiently small we can ensure that the convex combination $(1-s)u(\cdot, t+t_0)+su(\cdot, t)$ is close to $u(\cdot, t)$ in $C^2(\Omega)$ norm. By the uniform positive lower bound on the eigenvalues of $w_{ij}$, this ensures that the integrand in~\eqref{aij} remains positive definite, hence the equation for $v$ is parabolic.

Additionally, we see that for $x\in\ombdry$, $v$ satisfies
\begin{align*}
0&=G(x, \gradu(x, t))-G(x, \gradu(x, t+t_0))\notag\\
&=\left(\int_0^1\Gp(x, s\gradu(x, t)+(1-s)\gradu(x, t+t_0))ds\right)v_k\notag\\
&=: \alpha^kv_k
\end{align*}
By Theorem~\ref{unifobliquethm}, we have that $\Gp(x, \gradu(x, t))\nu_k\geq C>0$ for some $C$ uniform in $t$ and $x$. Thus as above, by choosing $t_0$ sufficiently small, we can ensure that $\alpha^k\nu_k\geq \frac{C}{2}>0$ and we see that $v$ satisfies a linear, uniformly oblique boundary condition.

Now following \cite[Section 6.2]{MR1918081} we obtain a translating solution of the same regularity as $u$, ie. a function $u^\infty(x, t)=u^\infty(x, 0)+C_\infty\cdot t$ for some constant $C_\infty$ that satisfies the equation~\eqref{floweqn}, such that $\lVert u-u^\infty\rVert_{C^k}\to 0$ as $t\to\infty$ for any $1\leq k\leq 4$. Thus we find
\begin{align*}
e^{C_\infty}f(x)=\lim_{t\to\infty}e^{\dot{u}^\infty}f(x)&=\lim_{t\to\infty}\det{(\hess{u^\infty-A(x, \gradu^\infty))}B^{-1}(x, \gradu^\infty)}\\
&=\lim_{t\to\infty}\det{(\hess{u-A(x, \gradu))}B^{-1}(x, \gradu)}\\
&=\lim_{t\to\infty}\det{[DT(x, \gradu)}]g(T(x, t)).
\end{align*}
We integrate both sides over $\Omega$, the $C^2$ estimates on $u$ along with bounds on the derivatives of $c$, $f$ and $g$ allow interchange of the integral and limit. Using the change of variables formula with the mass balance condition~\eqref{balance}, we obtain that $C_\infty=0$. Hence, $u^\infty(x, t) =u^\infty(x, 0)$ is independent of $t$, and by the convergence of $u$ to $u^\infty$ in the appropriate $C^k$ norms, we see it satisfies the desired elliptic equation~\eqref{elliptic eqn}, while Corollary~\ref{cor: T is onto} gives the desired mapping condition~\eqref{elliptic bdrycond}.

This completes the proof of the main theorem~\ref{thm: main theorem}.

\appendix
\appendixpage
\section{Construction of $\hstar$ and $h$}
Here we will show the construction of $\hstar$ and $h$ necessary to carry out the barrier arguments in the body of the paper. These constructions appear to be common knowledge in the field, but have not been explicitly written down to the knowledge of the author (c.f. \cite[Section 2]{MR2512204}).\\
Define the sets 
\begin{equation*}
\omepbdry=\left\{x\in\Omega \vert \dist(x, \ombdry)<\epsilon\right\}
\end{equation*}
and
\begin{equation*}
\omstarbdryepstar=\left\{y\in\omstar \vert \dist(y, \omstarbdry)<\epstar\right\}
\end{equation*}
\begin{thm}\label{thm: extended defining functions}
Assume that $\Omega$ and $\omstar$ are uniformly $c$ and $c^*$-convex with respect to each other. Then there exist $C^2$ functions $h$ on $\Omega$, and $\hstar$ on $\omstar$, and constants $\epsilon$, $\epstar$, $\delta_0$, $\delstar_0>0$ satisfying the following properties:
\begin{enumerate}
\item $\nabla h=\nu$ on $\ombdry$
\item $h<0$ on $\omepbdry$
\item $[D_{ij}h(x)-c^{l, k}c_{ij, l}(x, y)D_kh(x)]\xi_i\xi_j \geq \delta_0\lvert \xi \rvert^2$, $\forall x\in\omepbdry$, $y\in\omstar$, and $\xi \in\R^n$
\item $\nabla \hstar=\nustar$ on $\omstarbdry$
\item $\hstar<0$ on $\omstarbdryepstar$
\item $[D_{ij}\hstar(y)-c^{k, l}c_{l, ij}(x, y)D_k\hstar(y)]\xi_i\xi_j \geq \delstar_0\lvert \xi \rvert^2$, $\forall y\in\omstarbdryepstar$, $x\in\Omega$, and $\xi \in\R^n$
\end{enumerate}
\end{thm}

\begin{proof}
We will construct $h$, the construction for $\hstar$ is similar, but with the variables reversed.

Fix a $y\in\omstar$. Let $h(x)=Cd^2(x)-d(x)$, where $d(x)=\dist(x, \ombdry)$, and $C>0$ is a constant to be picked. We calculate that 
\begin{equation*}
D_ih=(2Cd-1)d_i
\end{equation*}
and
\begin{equation*}
D_{ij}h=(2Cd-1)d_{ij}+2Cd_id_j.
\end{equation*}
Clearly, $\nabla h(x)=-\nabla d(x)=\nu(x)$ for $x\in\ombdry$.

Now, fix a point $x\in\ombdry$, and take any $\xi\in\R^n$. We decompose $\xi=\tau(x)+a\nu(x)$ for some $a\in\R$ and $\tau(x)$ tangential to $\ombdry$ at $x$. Then, 
\begin{align*}
[D_{ij}h(x)-c^{l, k}c_{ij, l}(x, y)D_kh(x)]\xi_i\xi_j&=[-d_{ij}(x)+c^{l, k}c_{ij, l}(x, y)d_k(x)]\xi_i\xi_j\\
&\qquad+[2Cd_id_j]\xi_i\xi_j\\
&=I+II.
\end{align*}
Considering the matrix $d_id_j$, we see it has as a basis of eigenvectors: $-\nabla d=\nu$ with corresponding eigenvalue $\lvert \nabla d\rvert^2=1$, and $n-1$ orthogonal vectors, with corresponding eigenvalues of $0$. Thus, 
\begin{align*}
II&=2Cd_id_j(\tau_i\tau_j+2\tau_ia\nu_j+a^2\nu_i\nu_j)\\
&=0+0+2Ca^2
\end{align*}
and also
\begin{align*}
I&=[-d_{ij}(x)+c^{l, k}c_{ij, l}(x, y)d_k(x)](\tau_i\tau_j+2\tau_ia\nu_j+a^2\nu_i\nu_j)\\
&=[D_i\nu_j(x)-c^{l, k}c_{ij, l}(x, y)\nu_k(x)](\tau_i\tau_j+2a\tau_i\nu_j+a^2\nu_i\nu_j)\\
&=III+IV+V.
\end{align*}
By the assumption~\eqref{unifcconvex} of uniform $c$-convexity, we see that
\begin{equation*}
III\geq \delbar\lvert \tau\rvert^2
\end{equation*}
while
\begin{equation*}
V\geq -C_0a^2
\end{equation*}
for some $C_0>0$ depending only on $\lVert c\rVert_{C^3}$ and $\Omega$.
Finally, using Cauchy's inequality with an $\epsilon$, we obtain
\begin{equation*}
IV\geq -\frac{\delbar}{2}\lvert\tau\rvert^2 -\frac{2}{\delbar}a^2.
\end{equation*}
Combining these, we find for $C$ sufficiently large, whenever $x\in\ombdry$
\begin{align*}
[D_{ij}h(x)-c^{l, k}c_{ij, l}(x, y)D_kh(x)]\xi_i\xi_j&\geq \frac{\delbar}{2}\lvert\tau\rvert^2+(2C-C_0-\frac{2}{\delbar})a^2\\
&\geq \frac{\delbar}{2} (\lvert\tau\rvert^2+a^2)=\frac{\delbar}{2}\lvert\xi\rvert^2.
\end{align*}

Now, by taking $\epsilon$ sufficiently small, we can ensure $1\geq1-2Cd>\frac{1}{2}$ and that $h<0$ in $\omepbdry$. Using the compactness of $\ombdry$ and the continuity of $h$ and uniform continuity of $c^{l, k}c_{ij, l}$ in $x$ and $y$, and by taking $\epsilon$ smaller if necessary, we can ensure that for each $x\in\omepbdry$, there exists an $x_0\in\ombdry$ so that
\begin{equation*}
\begin{cases}
&\lvert [(-d_{ij}(x)+c^{l, k}c_{ij, l}(x, y)d_k(x))-(-d_{ij}(x_0)+c^{l, k}c_{ij, l}(x_0, y)d_k(x_0))]\xi_i\xi_j\rvert<\frac{\delbar}{16}\lvert\xi\rvert^2\\
&2C\vert[d_id_j(x)-d_id_j(x_0)]\xi_i\xi_j\rvert<\frac{\delbar}{16}\lvert\xi\rvert^2.
\end{cases}
\end{equation*}
Combining these, we find for $x\in\omepbdry$
\begin{align*}
&[D_{ij}h(x)-c^{l, k}c_{ij, l}(x, y)D_kh(x)]\xi_i\xi_j\\
&=(1-2Cd)[-d_{ij}(x)+c^{l, k}c_{ij, l}(x, y)d_k(x)]\xi_i\xi_j+[2Cd_id_j(x)]\xi_i\xi_j\\
&\geq (1-2Cd)[(-d_{ij}(x_0)+c^{l, k}c_{ij, l}(x_0, y)d_k(x_0))\\
&-\lvert (-d_{ij}(x)+c^{l, k}c_{ij, l}(x, y)d_k(x))-(-d_{ij}(x_0)+c^{l, k}c_{ij, l}(x_0, y)d_k(x_0))\rvert]\xi_i\xi_j\\
&\qquad+2C[d_id_j(x_0)-\vert(d_id_j(x)-d_id_j(x_0))\rvert]\xi_i\xi_j\\
&\geq (1-2Cd)(-d_{ij}(x_0)+c^{l, k}c_{ij, l}(x_0, y)d_k(x_0))\xi_i\xi_j+2Cd_id_j(x_0)\xi_i\xi_j\\
&-\frac{\delbar}{16}\lvert\xi\rvert^2-\frac{\delbar}{16}\lvert\xi\rvert^2\\
&=VI+VII-\frac{\delbar}{8}\lvert\xi\rvert^2.
\end{align*}
Now we have
\begin{equation*}
\begin{cases}
VI\geq(-d_{ij}(x_0)+c^{l, k}c_{ij, l}(x_0, y)d_k(x_0))\xi_i\xi_j,&(-d_{ij}(x_0)+c^{l, k}c_{ij, l}(x_0, y)d_k(x_0))\xi_i\xi_j<0\\
VI\geq\frac{1}{2}(-d_{ij}(x_0)+c^{l, k}c_{ij, l}(x_0, y)d_k(x_0))\xi_i\xi_j,&(-d_{ij}(x_0)+c^{l, k}c_{ij, l}(x_0, y)d_k(x_0))\xi_i\xi_j\geq 0
\end{cases}
\end{equation*}
so in either case we obtain
\begin{align*}
VI+VII-\frac{\delbar}{8}&\geq\frac{1}{2}(D_{ij}h(x_0)-c^{l, k}c_{ij, l}(x_0, y)D_kh(x_0))\xi_i\xi_j-\frac{\delbar}{8}\lvert\xi\rvert^2\\
&\geq(\frac{\delbar}{4}-\frac{\delbar}{8})\lvert\xi\rvert^2\\
&=\delta_0\lvert\xi\rvert^2
\end{align*}
for $\delta_0=\frac{\delbar}{8}$.
\end{proof}

We now make an auxillary construction, following \cite[Appendix A]{2008arXiv0805.3715B}.
\begin{cor}\label{cor: construction of G}
There exists a function $G(x, p)$ defined on $\omclose \times \nabla_xc(\omclose, \omstarclose)$ such that 
\begin{equation*}
\Gpp{l}\xi_k\xi_l\geq \deltil\lvert\xi\rvert^2
\end{equation*}
for all $\xi\in\R^n$, for some $\deltil>0$, depending only on $\omstar$ and $c$.\\
Also, for each $t\in[0, \tmax)$, $G(x, \gradu(x, t))=\hstarbar(Y(x, \gradu(x, t)))$ for all $x$ in some neighborhood of $\ombdry$ (possibly depending on $t$).\\
\end{cor}
\begin{proof}
Let $\phi$ be a $C^2$ function on $\R$ such that 
\[
\begin{cases}
\phi''\geq 0&\\
\phi(s) = \lvert s\rvert,&\lvert s\rvert > \frac{\epstar}{16}.
\end{cases}
\]
Note that we may choose $\phi$ so that $\lvert \phi'\rvert \leq 1$ everywhere.
Then define 
\begin{equation*}
G(x, p)=
\begin{cases}
\frac{\hstar(Y(x, p))+h_1(p)}{2}+\phi\left(\frac{\hstar(Y(x, p))-h_1(p)}{2}\right),&Y(x, p)\in \omstarbdryepstar\\
\hstar_1(p)&\text else
\end{cases}
\end{equation*}
where 
\begin{equation*}
\hstar_1(p)= \frac{1}{C_1}(\lvert p\rvert^2-K^2)
\end{equation*}
for some $C_1>0$ to be determined, and $K> \sup \lvert \gradu \rvert$ (which bounded independent of $t$ by Theorem~\ref{gradientbdd}).\\
Calculating, we find that for $Y(x, p)\in \omstarbdryepstar$,
\begin{align*}
\Gpp{l}&=\frac{1}{2}\left[1+\phi'\left(\frac{\hstar(Y(x, p))-\hstar_1(p)}{2}\right)\right]\Dpp{l}\hstar(Y(x, p))\\
&+\frac{1}{2}\left[1-\phi'\left(\frac{\hstar(Y(x, p))-\hstar_1(p)}{2}\right)\right]\Dpp{l}\hstar_1(p)\\
&+\frac{1}{4}\phi''\left(\frac{\hstar(Y(x, p))-\hstar_1(p)}{2}\right)(\Dp\hstar(Y(x, p))-\Dp\hstar_1(p))\\
&\qquad\cdot(\Dp[l]\hstar(Y(x, p))-\Dp[l]\hstar_1(p))\\
&=I_{kl}+II_{kl}+III_{kl}.
\end{align*}
Since $\phi''\geq 0$, we have that $III_{kl}$ is positive semi-definite.\\
By a simple calculation, 
\begin{align*}
\Dpp{l}\hstar(Y(x, p))&=(\Dp \hstar_m)Y^m_{p_l}=(\Dp \hstar_m) c^{m, l}=\notag\\
&=\hstar_{mn}Y^n_{p_k}c^{m, l}+\hstar_mc^{m, l}_{\ ,n}Y^n_{p_k}\notag\\
&=\hstar_{mn}c^{n, k}c^{m, l}+\hstar_mc^{m, l}_{\ ,n}c^{n, k}\notag\\
&=c^{n,k}c^{m,l}(\hstar_{mn}-\hstar_rc^{r, p}c_{p, nm})\notag\\
\end{align*}
so by Theorem~\ref{thm: extended defining functions},  
\begin{equation*}
\Dpp{l}\hstar(Y(x, p)) \xi_k \xi_l \geq \delstar_0 \sum_i \lvert c^{i, k}\xi_k \rvert^2 \geq 2\deltil \lvert \xi \rvert^2 \text{, some constant } \deltil.
\end{equation*}
Now, since $\lvert\phi'\rvert \leq 1$, we have that
\begin{equation*}
(I_{kl}+II_{kl})\xi_k\xi_l\geq\frac{1}{2}\Dpp{l}\hstar(Y(x, p)) \xi_k \xi_l\geq \deltil \lvert \xi \rvert^2
\end{equation*}
or
\begin{equation*}
(I_{kl}+II_{kl})\xi_k\xi_l\geq \frac{1}{2}\Dpp{l}\hstar_1(p) \xi_k \xi_l\geq \frac{1}{C_1}\lvert\xi\rvert^2
\end{equation*}
and hence, $\Gpp{l}$ is uniformly positive definite for $Y(x, p)\in \omstarbdryepstar$. For $(x, p)\in Int(Y^{-1}(\omstarbdryepstar))$, $G=\hstar_1$ so $\Gpp{l}$ is clearly also uniformly positive definite there. We will conclude the proof of positivity by showing that $\hstar(Y(x, p))<\hstar_1(p)$ on $\omtilbdry \cap Int(\omclose \times \nabla_xc(\omclose, \omstarclose))$, hence $G=\hstar_1$ on a neighborhood of that set.
We claim that if $(x_0, p_0)\in\omtilbdry \cap Int(\omclose \times \nabla_xc(\omclose, \omstarclose))$, then $d(Y(x_0, p_0), \omstarbdry)=\epstar$. Indeed, we can just use the continuity $Y$ and the distance to the boundary, along with taking sequences lying in $\omtil$ and $Int(\omclose \times \nabla_xc(\omclose, \omstarclose))\setminus \omtil$ that converge to $(x_0, p_0)$.
With this, we can see that $\hstar(Y(x_0, p_0))=(C(\epstar)^2-\epstar)=\epstar(C\epstar-1)<-\frac{\epstar}{2}$ for $\epstar$ small enough depending only on $C$. Now by choosing $\frac{1}{C_1}$ sufficiently small, we can ensure that $-\frac{\epstar}{2}<-\frac{K^2}{C_1}\leq\hstar_1(p)$ on $\omtilbdry \cap Int(\omclose \times \nabla_xc(\omclose, \omstarclose))$, giving us the desired claim. This also shows that $G$ is $C^2$.

Finally, for a fixed $t$, $\hstar(Y(x, \gradu(x, t)))=0>\hstar_1(\gradu(x, t))$ for $x\in\ombdry$. Thus, by the continuity of $\gradu$ and the compactness of $\ombdry$, we can find a neighborhood of $\ombdry$ on which $G(x, \gradu(x, t))=\hstar(Y(x, \gradu(x, t)))=\hstarbar(Y(x, \gradu(x, t)))$.
\end{proof}
\nocite{MR2512204}
\nocite{MR2008688}
\nocite{MR2034229}
\bibliographystyle{plain}
\bibliography{mybiblio}

\end{document}